\documentclass{amsart}
\newtheorem{theorem}{Theorem}[section]
\newtheorem{lemma}[theorem]{Lemma}
\newtheorem{corollary}[theorem]{Corollary}

\theoremstyle{definition}

\theoremstyle{remark}
\newtheorem{remark}[theorem]{Remark}

\numberwithin{equation}{section}

\numberwithin{theorem}{section}
\numberwithin{equation}{section}

\usepackage[all]{xy}
\usepackage{amssymb,amsmath,stmaryrd}
\usepackage{amscd}
\usepackage{braket,amsfonts,amsopn}
\usepackage{bm}
\usepackage{makecell,multirow,diagbox}
\usepackage{mathrsfs}
\usepackage{graphicx,epstopdf}
\usepackage{subfigure}
\usepackage[pdftex,colorlinks]{hyperref}

\DeclareMathOperator{\curl}{curl}

\DeclareMathOperator{\tr}{tr}

\newcommand{\Th}{\mathcal{T}_{h}}

\newcommand{\CV}{\mathcal{V}}

\newcommand{\ab}[2]{\langle#1,#2\rangle}

\newcommand{\db}[1]{\{\!\!\{#1\}\!\!\}}
\newcommand{\lr}[1]{\llbracket#1\rrbracket}
\newcommand{\vertiii}[1]{|\!|\!|#1|\!|\!|}

\usepackage{verbatim}

\DeclareFontFamily{U}{mathx}{\hyphenchar\font45}
\DeclareFontShape{U}{mathx}{m}{n}{
      <5> <6> <7> <8> <9> <10>
      <10.95> <12> <14.4> <17.28> <20.74> <24.88>
      mathx10
      }{}
\DeclareSymbolFont{mathx}{U}{mathx}{m}{n}
\DeclareFontSubstitution{U}{mathx}{m}{n}
\DeclareMathAccent{\widecheck}{0}{mathx}{"71}

\title{An Extended Galerkin analysis in finite element exterior calculus}

\begin{document}
\bibliographystyle{amsplain}
    

\author{Qingguo Hong, Yuwen Li, Jinchao Xu}
\address{Department of Mathematics, The Pennsylvania State University, University Park, PA 16802.}
\email{huq11@psu.edu,yuwenli925@gmail.com,jxx1@psu.edu}
\thanks{Y.~Li is the corresponding author}
\thanks{
J.~Xu was supported by Center for Computational Mathematics and Applications, and the Verne M. William Professorship Fund from 
the Pennsylvania State University}


\date{\today}
\begin{abstract}
For the Hodge--Laplace equation in finite element exterior calculus, we introduce several families of discontinuous Galerkin methods in the extended Galerkin framework. For contractible domains, this framework utilizes seven fields and provides a unifying inf-sup analysis with respect to all discretization and penalty parameters. It is shown that the proposed methods can be hybridized as a reduced two-field formulation.
\end{abstract}
\maketitle
\markboth{Q.~Hong,\ Y.~Li,\ J.~Xu}{Extended Galerkin Analysis in FEEC}


\section{Introduction}
Finite element exterior calculus (FEEC) is a powerful and elegant framework unifying numerical analysis on the de Rham complex, or more generally,  closed Hilbert complexes. 
In recent decades, theory of FEEC has been under intensive investigation and applied to a variety of problems such as vector Laplacian \cite{ArnoldFalkWinther2006}, Stokes \cite{FalkNeilan2013,ChristiansenHu2018}, Maxwell \cite{Hiptmair2002}, and elasticity equations \cite{ArnoldFalkWinther2007,ArnoldAwanouWinther2008}. For a thorough introduction to FEEC, readers are referred to  \cite{ArnoldFalkWinther2006,ArnoldFalkWinther2010,Arnold2018} and references therein. As a model problem in FEEC, the continuous Galerkin/Arnold--Falk--Winther (AFW) method for the following Hodge--Laplace equation (without harmonic forms)
\begin{equation}\label{HodgeLaplaceSimple}
    (d\delta+\delta d)u=f
\end{equation}
has been discussed in many aspects, see, e.g., \cite{ArnoldFalkWinther2006,ChristiansenWinther2008,FalkWinther2014,Licht2019} for  commuting projections and a priori error estimates, \cite{DemlowHirani2014,Demlow2017,ChenWu2017,Li2019SINUM,Li2020MCOM} for a posteriori error estimates and adaptive algorithms, \cite{GilletteHolstZhu2017,ArnoldChen2017} for time-dependent problems, and \cite{ArnoldAwanou2014,GilletteKloefkorn2019,DiPietroDroniouRapetti2020} for FEEC on cubical and polyhedral meshes.

The discontinuous
Galerkin (DG) methods can be traced back to the late 1960s
\cite{Lions1968, aubin1970approximation}. DG methods have achieved great success in purely convective or convection-dominated problems in fluids, see, e.g., \cite{LasaintRaviart1974,ReedHill1973,CockburnShu1989,cockburn2000development} and references therein. In recent decades, DG methods have also been applied to purely
elliptic problems, see, e.g., interior penalty methods \cite{douglas1976interior,babuvska1973nonconforming, wheeler1978elliptic,
arnold1982interior}, enriched Galerkin methods \cite{BeckerBurmanHansboLarson2003,SunLiu2009}, local DG methods \cite{cockburn1998local}, and hybridized DG methods \cite{CarreroCockburnSchotzau2006,CockburnGopalakrishnanLazarov2009}. A unified DG analysis for elliptic problems could be found in \cite{ArnoldBrezziCockburnMarini2001}. To achieve
consistency and stability, DG methods often utilize additional finite element spaces on the boundaries of elements. In most existing DG methods, only one boundary discrete space (Lagrangian multiplier
space) is introduced, cf.~ \cite{brezzi2000discontinuous,hong2017unified}. Recently, an extended Galerkin (XG) framework for Poisson's equation was presented in \cite{HongWuXu2021}. This framework makes use of two auxiliary finite element spaces (check spaces) on element boundaries and unifies several classes of popular DG methods in the literature. The XG method was further generalized to linear elasticity equations in \cite{HongHuMaXu2020}.

To the best of our knowledge, a detailed analysis of DG methods in FEEC is still \emph{missing}. Very recently, the work  \cite{AwanouFabienGuzmanStern2020} derived hybridization and postprocessing technique for the AFW mixed method in FEEC and discussed possible extensions to hybridized DG methods. In addition, there are several relevant works on non-standard discretizations of the Maxwell equations, see, e.g., DG methods \cite{HoustonPerugiaSchtzau2004,HoustonPerugiaSchtzau2005,NguyenPeraireCockburn2011,li2013hybridizable}, weak Galerkin methods \cite{MuWangYeZhang2015}, and virtual
  element methods \cite{da2017virtual}. Comparing to Maxwell equations, the DG discretization of the mixed Hodge Laplacian is expected to involve more unknown variables, numerical fluxes and stabilization terms.

In this work, we present several DG methods for the Hodge--Laplace equation in a unified XG framework, which employs three unknowns approximating $u,$ $du$, $\delta u$ as well as four ``check'' variables defined on element boundaries. Similarly to the XG method for Poisson's equation \cite{HongWuXu2021}, the seven-field XG formulation is motivated by the so called Nitsche’s trick \cite{Nitsche1971,Nitsche1972}. In fact, the solution $u$ of the Hodge Laplace equation  \eqref{HodgeLaplaceSimple} satisfies two types of inter-element continuity conditions based on $d$ and $\delta$. In the existing literature,  the classical AFW mixed method essentially utilizes one of them. Working in totally discontinuous Galerkin subspaces, the XG method penalizes the inter-element continuity w.r.t. both $d$ and $\delta$ by check variables. Therefore, the proposed XG method stays in an intermediate stage between AFW conforming methods. In addition, it is shown that the solution of our XG method converges to the classical AFW mixed method as penalty parameters approach suitable limits.

By eliminating one or more fields, the seven-field XG formulation could be further reduced to a variety of more compact DG methods. In the extreme case, we hybridize the XG method to obtain a condensed system with only two unknown numerical fluxes. It is well known that the Hodge Laplacian covers the Poisson's equation and vector Laplacian as special cases. Correspondingly, the proposed seven-field XG method reduces to the original four-field XG method in \cite{HongWuXu2021} for the Poisson equation in primal and mixed forms. It also leads to several  new families of DG methods for the vector Laplacian.

The XG framework also facilitates a transparent convergence analysis for DG discretizations of the Hodge--Laplace equation. In particular, we derive the quasi-optimal error estimate of the XG method through constructing a discrete inf-sup condition uniform w.r.t.~all discretization and penalty parameters. In doing so, several fields on element boundaries are eliminated and the remaining variables are carefully grouped to obtain a  standard perturbed saddle point system \eqref{XGreduced}. It turns out that the resulting saddle system could be analyzed block-wise using the classical Babuska--Brezzi inf-sup theory, cf.~\cite{BoffiBrezziFortin2013}.

The rest of this paper is organized as follows. In Section \ref{SectionPreliminaries}, we introduce notation and preliminaries in FEEC. In Section \ref{SectionXGformulation}, we derive the XG formulations for the Hodge--Laplace equation. Section \ref{SectionInfsup} is then devoted to the analysis of continuous and discrete inf-sup conditions. In Section \ref{SectionConnection}, we discuss the
relationship between XG methods and the classical mixed methods. In Section \ref{SectionHybridization}, a hybridized version of the proposed XG method is proposed. Finally, we restate XG methods for the Hodge--Laplace equation in the language of vector calculus in Section \ref{SectionExamples}.

\section{Preliminaries}\label{SectionPreliminaries}
In this section, we follow the convention in \cite{ArnoldFalkWinther2006,ArnoldFalkWinther2010} to present minimal  preliminaries in FEEC for developing XG methods with differential forms.
\subsection{De Rham complex}
Let $\Omega\subset\mathbb{R}^{n}$ be a bounded Lipschitz polyhedron. For simplicity of presentation, throughout the rest of this paper, we assume that the cohomology of $\Omega$ is trivial,  e.g., $\Omega$ is contractible. 

\subsubsection{Continuous de Rham Complex}
Given an integer $k\in\{0,1,\ldots,n\}$, let $S^n_k$ denote the set of increasing multi-indices in the form $\alpha=(\alpha_1,\ldots,\alpha_k)\in\mathbb{N}^k$ with $1\leq\alpha_1<\cdots<\alpha_k\leq n$. Let $\Lambda^{k}(\Omega)$ be the space of smooth differential $k$-forms  written as
\begin{equation*}
    \omega=\sum_{\alpha\in S_k^n}\omega_\alpha dx^{\alpha_{1}}\wedge\cdots\wedge dx^{\alpha_{k}},
\end{equation*}
where each coefficient $\omega_{\alpha}\in C^{\infty}(\overline{\Omega})$, and $\wedge$ is the wedge product. The space $\Lambda^{k}(\Omega)$ is naturally endowed with the $L^{2}$-inner product $(\cdot,\cdot)$ given as
\begin{equation*}
    (\omega,\eta):=\sum_{1\leq\alpha_{1}<\cdots<\alpha_{k}\leq n}\int_\Omega\omega_\alpha\eta_\alpha dx,\quad\omega,\eta\in\Lambda^k(\Omega).
\end{equation*}
The induced $L^2$-norm is $\|\cdot\|$. We further denote the $L^2$-inner product and $L^2$-norm on a Lipschitz submanifold $U\subseteq\Omega$ as $(\cdot,\cdot)_U$ and $\|\cdot\|_U$, respectively.
Let $d^k: \Lambda^{k}(\Omega)\rightarrow\Lambda^{k+1}(\Omega)$ denote the exterior derivative for differential forms, that is,
\begin{equation*}
    d^k\omega=\sum_{\alpha\in S_k^n}\sum_{i=1}^n\frac{\partial\omega_\alpha}{\partial x_i}dx^i\wedge dx^{\alpha_{1}}\wedge\cdots\wedge dx^{\alpha_{k}}.
\end{equation*}
The space $L^{2}\Lambda^{k}(\Omega)$ is the completion of $\Lambda^{k}(\Omega)$ w.r.t.~the norm $\|\cdot\|$, which is simply the space of $k$-forms with $L^2(\Omega)$ coefficients. 
The derivative $d^k$ can be extended as a weak derivative $d^k: L^2\Lambda^{k}(\Omega)\rightarrow L^2\Lambda^{k+1}(\Omega)$, which is a densely defined unbounded operator with domain 
\begin{equation*}
    H\Lambda^{k}(\Omega):=\big\{\omega\in L^{2}\Lambda^{k}(\Omega): d^k\omega\in L^{2}\Lambda^{k+1}(\Omega)\big\}.
\end{equation*}
We shall also make use of $H^{s}\Lambda^{k}(\Omega)$, the Sobolev space of differential $k$-forms whose coefficients are in $H^{s}(\Omega)$. Due to $d^k\circ d^{k-1}=0$, the following sequence 
\begin{equation}\label{deRham}
\begin{CD}
H\Lambda^0(\Omega)@>d^0>>H\Lambda^1(\Omega)@>d^1>>\cdots@>d^{n-2}>>H\Lambda^{n-1}(\Omega)@>d^{n-1}>> H\Lambda^{n}(\Omega)
\end{CD}
\end{equation}
is called the de Rham complex.
Let $\mathfrak{Z}^k:=N(d^k)$ be the kernel of $d^k$,  $\mathfrak{B}^k:=R(d^{k-1})$ be the range of $d^{k-1}$, and $\mathfrak{H}^{k}:=\mathfrak{Z}^k\cap\mathfrak{B}^{k\perp}$ be the space of harmonic $k$-forms. Here $\perp$ is the operation of taking $L^2$-orthogonal complement in $H\Lambda^k(\Omega).$  Due to the assumption on $\Omega,$ we have 
\begin{equation}
\mathfrak{H}^k=
\left\{\begin{aligned}
    \mathbb{R},\quad k=0,\\
    \{0\},\quad k\geq1,
\end{aligned}\right.
\end{equation}
and thus $\mathfrak{B}^k=\mathfrak{Z}^k$ if $k\geq1.$ It then follows that
\begin{equation}\label{Helm}
  H\Lambda^k(\Omega)=\mathfrak{Z}^k\oplus\mathfrak{Z}^{k\perp}=\mathfrak{B}^k\oplus\mathfrak{Z}^{k\perp},\quad k\geq1.
\end{equation}
In addition, there exists an absolute constant $c_P>0$ relying on $\Omega$ such that 
\begin{equation}\label{Poincare}
    \|v\|\leq c_P\|d^kv\|,\quad\forall v\in\mathfrak{Z}^{k\perp}.
\end{equation}

For each index $k$, the Hodge star $\star: L^2\Lambda^k(\Omega)\rightarrow L^2\Lambda^{n-k}(\Omega)$ is a linear isomorphism determined by the equation
\begin{equation*}
    \int_{\Omega}\omega\wedge\mu=(\star\omega,\mu),\quad\forall\mu\in L^2\Lambda^{n-k}(\Omega).
\end{equation*}
It is well-known that for $\omega, \mu\in L^2\Lambda^k(\Omega)$,
\begin{subequations}\label{star0}
    \begin{align}
    &\star\star\omega=(-1)^{k(n-k)}\omega,\label{star1}\\
    &(\star\omega,\star\mu)=(\omega,\mu).\label{star2}
    \end{align}
\end{subequations}
With the help of $\star,$ the coderivative $\delta^k: H^1\Lambda^k(\Omega)\rightarrow L^2\Lambda^{k-1}(\Omega)$ is then defined as
\begin{equation}\label{delta}
  \star\delta^k\omega=(-1)^kd^{n-k}\star\omega.
\end{equation}
Given Lipschitz manifolds $M\subseteq\Omega_0\subseteq\Omega$, let $i_{M,\Omega_0}: M\rightarrow\Omega_0$ be the inclusion, and $\tr_{M,\Omega_0}=i_{M,\Omega_0}^{*}$ be the pullback for differential forms based on $i_{M,\Omega_0}$, i.e., the trace operator on $M$. We denote $\tr_{M}=\tr_{M,\Omega}$ and may suppress the subscript $M$ provided relevant domains are  clear from the context.  
It is well known that $d^k$ and $\delta^{k+1}$ are related by the Stokes' formula
\begin{equation}\label{IntegrationByParts}
(d^k\omega,\mu)_U=(\omega,\delta^{k+1}\mu)_U+\int_{\partial U}\tr_{\partial U,U}(\omega\wedge\star\mu),
\end{equation}
where $\omega\in H^1\Lambda^{k}(U),~\mu\in H^1\Lambda^{k+1}(U)$, and $\partial U$ is set to be \emph{outward orientated}.
The Sobolev space of $k$-forms with essential boundary condition is 
\begin{equation*}
    \mathring{H}\Lambda^{k}(\Omega):=\big\{\omega\in H\Lambda^{k}(\Omega): \tr_{\partial\Omega}\omega=0\big\}.
\end{equation*}
Such spaces form a de Rham complex with essential boundary condition
\begin{equation}\label{deRhamessential}
\begin{CD}
\mathring{H}\Lambda^0(\Omega)@>d^0>>\mathring{H}\Lambda^1(\Omega)@>d^1>>\cdots@>d^{n-2}>>\mathring{H}\Lambda^{n-1}(\Omega)@>d^{n-1}>> \mathring{H}\Lambda^{n}(\Omega)
\end{CD}
\end{equation}

The coderivative $\delta^k$ could be extended as a closely defined and unbounded operator $\delta^{k}: L^2\Lambda^{k}(\Omega)\rightarrow L^2\Lambda^{k-1}(\Omega)$ with domain
\begin{align*}&\mathring{H}^{*}\Lambda^{k}(\Omega):=\star\mathring{H}\Lambda^{n-k}(\Omega)\\
  &=\{\omega\in L^{2}\Lambda^{k}(\Omega): \delta^k\omega\in L^{2}\Lambda^{k-1}(\Omega),\ \tr_{\partial\Omega}\star
\omega=0\}.
\end{align*}
As a result of $\delta^k\circ\delta^{k+1}=0$, we have the following exact sequence 
\begin{equation}\label{dualdeRham}
\begin{CD}
\mathring{H}^*\Lambda^0(\Omega)@<\delta^1<<\mathring{H}^*\Lambda^1(\Omega)@<\delta^2<<\cdots@<\delta^{n-1}<<\mathring{H}^*\Lambda^{n-1}(\Omega)@<{\delta^n}<< \mathring{H}^*\Lambda^{n}(\Omega),
\end{CD}
\end{equation}
which is 
the dual complex of \eqref{deRham}.

\subsubsection{Discrete de Rham Complex}
In the following, we present the discrete analogue of the de Rham complex \eqref{deRham}. Let $\Th$ be a conforming partition of $\Omega$, which is shape-regular in the sense that
\begin{equation*}
    \max_{T\in\mathcal{T}_h}(r_T/\rho_T):=C_{\text{mesh}}<\infty,
\end{equation*}
where $r_T$, $\rho_T$ are radii of circumscribed and inscribed spheres of $T$, and  $C_{\text{mesh}}$ is an absolute constant.
Given a Lipschitz manifold $U$ and an integer $r\geq0$, define \begin{equation*}
    \mathcal{P}_r\Lambda^k(U):=\{\omega\in L^2\Lambda^k(U): \omega_\alpha\in\mathcal{P}_r(U)\text{ for all }\alpha\in S^n_k\},
\end{equation*}
which is the space of $k$-forms on $U$ with polynomial coefficients of degree at most $r$.
Let
$\kappa: L^2\Lambda^{k+1}(U)\rightarrow L^2\Lambda^k(U)$ denote the interior product for differential forms.
We shall also make use of the space of $k$-forms on $U$ with incomplete polynomial coefficients 
\begin{equation*}
    \mathcal{P}^-_r\Lambda^k(U):=\mathcal{P}_{r-1}\Lambda^k(U)+\kappa\big(\mathcal{P}_{r-1}\Lambda^{k+1}(U)\big).
\end{equation*}
There exist two families of conforming finite element subspaces of $H\Lambda^k(\Omega)$:
\begin{align*}
    &\mathcal{P}_r\Lambda^k(\Th):=H\Lambda^k(\Omega)\cap\prod_{T\in\Th}\mathcal{P}_r\Lambda^k(T),\\
    &\mathcal{P}^-_r\Lambda^k(\Th):=H\Lambda^k(\Omega)\cap\prod_{T\in\Th}\mathcal{P}^-_r\Lambda^k(T).
\end{align*}
Under essential boundary conditions, we define
\begin{align*}
    &\mathring{\mathcal{P}}_r\Lambda^k(\Th):=\mathring{H}\Lambda^k(\Omega)\cap\mathcal{P}_r\Lambda^k(\mathcal{T}_h),\\
    &\mathring{\mathcal{P}}^-_r\Lambda^k(\Th):=\mathring{H}\Lambda^k(\Omega)\cap\mathcal{P}^-_r\Lambda^k(\mathcal{T}_h).
\end{align*}
For each index $k$, we may choose  $V^k_h=\mathcal{P}_{r_k}\Lambda^k(\Th)$ or $\mathcal{P}^-_{r_k}\Lambda^k(\Th)$. In addition, $\{V_h^k\}_{k=0}^n$ is assumed to form a discrete de Rham complex 
\begin{equation}\label{discretedeRham}
\begin{CD}
V_h^0@>d^0>>V_h^1@>d^1>>\cdots@>d^{n-2}>>V_h^{n-1}@>d^{n-1}>> V_h^{n},
\end{CD}
\end{equation}
which could be connected with \eqref{deRham} via commuting interpolations, see \eqref{FEpairs} and \cite{ArnoldFalkWinther2006,ArnoldFalkWinther2010} for more details. In the discrete level, \begin{equation*}
    \mathring{V}_h^k:=\mathring{H}\Lambda^k(\Omega)\cap V_h^k,\quad0\leq k\leq n
\end{equation*}
form a discrete complex with essential boundary conditions.
Similarly to the continuous case, let $\mathfrak{Z}_h^k:=N(d^k|_{V_h^k})$,  $\mathfrak{B}_h^k:=d^{k-1}(V_h^{k-1})$, and  $\mathfrak{H}_h^{k}:=\mathfrak{Z}_h^k\cap\mathfrak{B}_h^{k\perp}$. The space of discrete harmonic $k$-forms $\mathfrak{H}_h^{k}$ is isomorphic to $\mathfrak{H}^k$. In particular,
\begin{equation}
\mathfrak{H}_h^k=
\left\{\begin{aligned}
    \mathbb{R},\quad k=0,\\
    \{0\},\quad k\geq1.
\end{aligned}\right.
\end{equation}
For $k\geq1,$
it is straightforward to obtain
\begin{equation}\label{discreteHodge}
  V_h^k=\mathfrak{B}_h^k\oplus\mathfrak{Z}_h^{k\perp}=d^{k-1}V_h^{k-1}\oplus\mathfrak{Z}_h^{k\perp},
\end{equation}
i.e., the discrete Hodge decomposition. In addition, there exists an absolute constant $c_{dP}>0$ dependent on $\Omega$ and $C_{\text{mesh}}$ such that 
\begin{equation}\label{discretePoincare}
    \|v_h\|\leq c_{dP}\|d^kv_h\|,\quad\forall v_h\in\mathfrak{Z}_h^{k\perp}.
\end{equation}
In the literature, \eqref{discretePoincare} is known as the discrete Poincar\'e inequality and the proof could be found in e.g.,  \cite{ArnoldFalkWinther2006}.

%
Throughout the rest of this paper, we may suppress the super-index of $\delta^k$, $d^k$, $L^2\Lambda^k(\Omega)$,  $H\Lambda^k(\Omega)$ etc. for a fixed $k$ and 
adopt the notation
\begin{align*}
    &d=d^k,\quad \delta=\delta^k,\quad d^-=d^{k-1},\quad\delta^+=\delta^{k+1},\\
    &V^-=H\Lambda^{k-1}(\Omega),\quad V=H\Lambda^{k}(\Omega),\quad W^+=L^2\Lambda^{k+1}(\Omega).
\end{align*}
Similar notation will be used in the discrete level. The $V^-$- and $V$-norms are 
\begin{align*}
    &\|\tau\|_{V^-}^2=\|\tau\|^2+\|d^-\tau\|^2,\quad \tau\in V^-,\\
    &\|v\|_{V}^2=\|v\|^2+\|dv\|^2,\quad v\in V.
\end{align*}

\subsection{Hodge Laplacian and approximation}
Given an index $k\geq1$ and data $f\in L^2\Lambda^k(\Omega)$,
we consider the the Hodge--Laplace equation
\begin{equation}\label{HodgeLaplace}
    (d^-\delta+\delta^+ d)u=f.
\end{equation}
The solution $u$ is contained in the space $H\Lambda^{k}(\Omega)\cap \mathring{H}^*\Lambda^k(\Omega),$ which is the intersection of domains of $d$ and $\delta$. Therefore the primal variational formulation for \eqref{HodgeLaplace} is to find $u\in H\Lambda^{k}(\Omega)\cap \mathring{H}^*\Lambda^k(\Omega)$ such that
\begin{equation}\label{varHL}
  (du,dv)+(\delta  u,\delta v)=(f,v),\quad\forall v\in H\Lambda^{k}(\Omega)\cap \mathring{H}^*\Lambda^k(\Omega).
\end{equation}
Although \eqref{varHL} is well-posed, a practical finite element subspace of $H\Lambda^{k}(\Omega)\cap \mathring{H}^*\Lambda^k(\Omega)$ is not available yet. To remedy this situation, Arnold, Falk, and Winther \cite{ArnoldFalkWinther2010} considered the mixed formulation of \eqref{HodgeLaplace}: Find $(\sigma,u)\in V^-\times V$, such that
\begin{subequations}\label{mixedHodgeLaplace0}
\begin{align}
(\sigma,\tau)+(d^-\tau,u)&=0,\quad&&\tau\in V^-,\label{mHL1}\\
(d^-\sigma,v)-(du,dv)&=-(f,v),&& v\in V.
\end{align}
\end{subequations}
From \eqref{mHL1},
it could be observed that $\sigma=-\delta u$. The well-posedness of the variational formulation \eqref{mixedHodgeLaplace0} was verified in \cite{ArnoldFalkWinther2006}.

To discretize the Hodge--Laplace equation \eqref{mixedHodgeLaplace0}, the finite element pair $V_h^-\times V_h=V_h^{k-1}\times V_h^k$
must be carefully chosen such that a discrete inf-sup stability is satisfied. In this paper, we choose the Arnold--Falk--Winther finite element space
\begin{equation}\label{FEpairs}
    V_h^-=\left\{\begin{aligned}
        &\mathcal{P}_{r+1}\Lambda^{k-1}(\Th),\\
        &\mathcal{P}^-_{r+1}\Lambda^{k-1}(\Th)
    \end{aligned}\right\},\quad V_h=\left\{\begin{aligned}
        &\mathcal{P}^-_{r+1}\Lambda^{k}(\Th),\\
        &\mathcal{P}_{r}\Lambda^{k}(\Th)
    \end{aligned}\right\}.
\end{equation}
The AFW mixed method  for  \eqref{mixedHodgeLaplace} is to find $(\sigma^c_{h},u^c_{h})\in V_{h}^-\times V_{h}$, such that
\begin{equation}\label{AFW}
\begin{aligned}
(\sigma^c_{h},\tau_h)+(d^-\tau_h,u^c_{h})&=0,&&\tau_h\in V_{h}^-,\\
(d^-\sigma^c_{h},v_h)-(du^c_{h},dv_h)&=-(f,v_h),&&v_h\in V_{h}.
\end{aligned}
\end{equation}
In \cite{ArnoldFalkWinther2006,ArnoldFalkWinther2010}, the discrete inf-sup condition and a priori error estimates of \eqref{AFW} were established.

\section{Extended Galerkin Formulation}\label{SectionXGformulation}
In this section, we derive an XG framework unifying several DG methods for the Hodge Laplacian \eqref{HodgeLaplace}. Those DG methods converge to the classical AFW mixed method \eqref{AFW} in limiting cases.

\subsection{DG Notation}
In a partition $\Th,$ let $\mathcal{E}_h$ denote the set of faces (of dimension $n-1$), and $\mathcal{E}_h^\partial\subset\mathcal{E}_h$ be the collection of faces that are contained in $\partial\Omega$. The set of interior faces is then $\mathcal{E}_h^o:=\mathcal{E}_h\backslash\mathcal{E}_h^\partial$. Given subsets $\mathcal{E}\subseteq\mathcal{E}_h$ and $\mathcal{T}\subseteq\Th,$ let
\begin{align*}
    &H^{1}\Lambda^{k}(\mathcal{T}):=\prod_{T\in\mathcal{T}}H^{1}\Lambda^{k}(T),\quad L^2\Lambda^k(\mathcal{E}):=\prod_{E\in\mathcal{E}}L^2\Lambda^k(E),\\
    &\tilde{\mathcal{P}}_r\Lambda^{k}(\mathcal{T}):=\prod_{T\in\mathcal{T}}\mathcal{P}_r\Lambda^{k}(T),\quad \mathcal{P}_r\Lambda^{k}(\mathcal{E}):=\prod_{E\in\mathcal{E}}\mathcal{P}_r\Lambda^{k}(E),\\
    &\tilde{\mathcal{P}}^-_r\Lambda^{k}(\mathcal{T}):=\prod_{T\in\mathcal{T}}\mathcal{P}^-_r\Lambda^{k}(T).
\end{align*}
Let $\langle\cdot,\cdot\rangle_E$ be the $L^2$-inner product of differential forms on $E$. We adopt the notation
\begin{equation*}
    \ab{\cdot}{\cdot}_\mathcal{E}=\sum_{E\in\mathcal{E}}\langle\cdot,\cdot\rangle_E,\quad\ab{\cdot}{\cdot}_{\partial\mathcal{T}}=\sum_{T\in\mathcal{T}}\langle\cdot,\cdot\rangle_{\partial T},\quad\ab{\cdot}{\cdot}=\ab{\cdot}{\cdot}_{\mathcal{E}_h}.
\end{equation*}
The $L^2$-norm based on $\langle\cdot,\cdot\rangle_{\mathcal{E}}$ is denoted as $\|\cdot\|_\mathcal{E}$. 
Each face $E\in\mathcal{E}_h$ is assigned with a unit norm $\nu_E$. Given an interior face $E\in\mathcal{E}_h^o,$ let $T_E^+$, $T_E^-$ denote elements in $\mathcal{T}_h$ sharing $E$ as a face such that $\nu_E$ is pointing out of $T_E^+$. In addition, $\nu_E$ is chosen to be outward and $T_E\in\mathcal{T}_h$ denotes the element containing $E$ if $E\subset\partial\Omega$ is a boundary face, see Figure \ref{elementpair}. 
\begin{figure}[ptb]
\centering
  \begin{tabular}[c]{cccc}%
  \subfigure[]{\includegraphics[width=5.5cm,height=5cm]{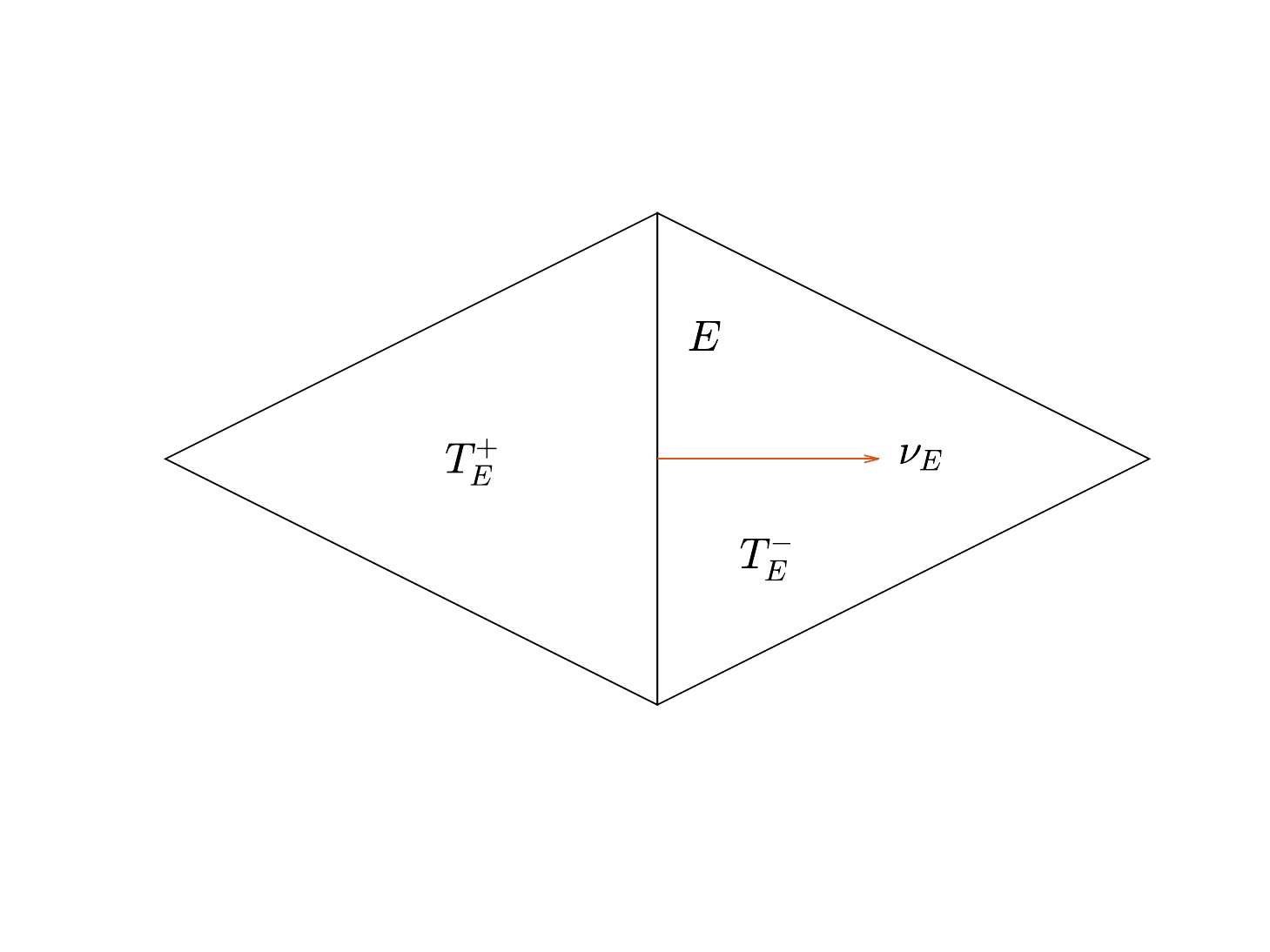}}
  \subfigure[]{\includegraphics[width=5cm,height=4.5cm]{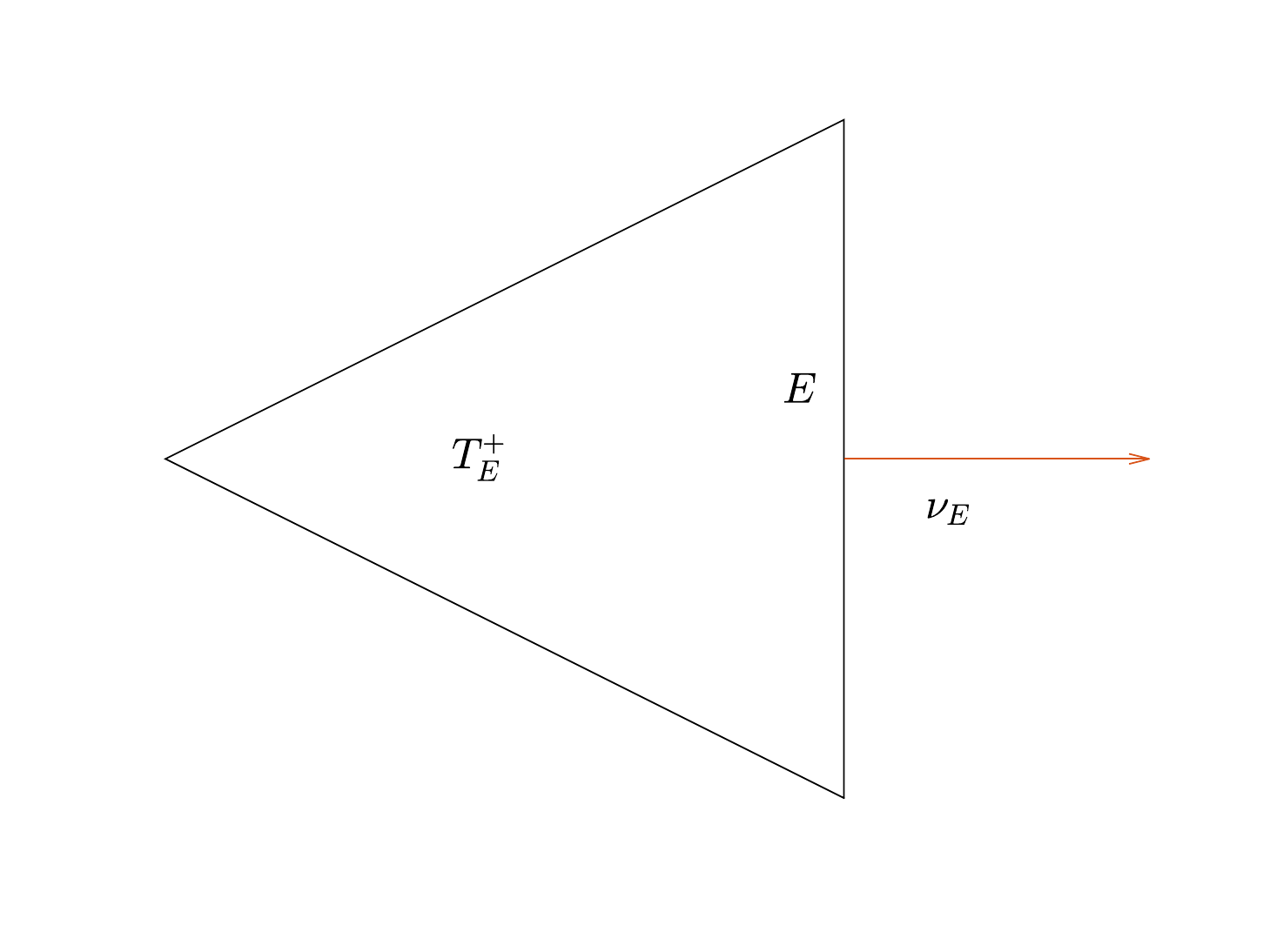}}
  \end{tabular}
  \caption{An interior element pair and a boundary element.}
  \label{elementpair}
\end{figure}
On a face $E\in\mathcal{E}_h$, the face Hodge star $\bar{\star}: L^2\Lambda^k(E)\rightarrow L^2\Lambda^{n-k}(E)$ is a linear isometry such that for $\eta_1, \eta_2\in L^2\Lambda^k(E)$
and $\eta_3\in L^2\Lambda^{n-k-1}(E)$, 
\begin{subequations}\label{barstar}
    \begin{align}
    &\ab{\bar{\star}\eta_1}{\bar{\star}\eta_2}_E=\ab{\eta_1}{\eta_2}_E,\label{barstar3}\\
    &\int_E\eta_1\wedge\eta_3=\ab{\bar{\star}\eta_1}{\eta_3}_E\label{barstar2}.
    \end{align}
\end{subequations}

Given $\omega\in H^{1}\Lambda^k(\Th)$ and $E\in\mathcal{E}_h$, we introduce the trace average $\db{\tr\omega}\in L^2\Lambda^k(\mathcal{E}_h)$ and trace jump $\lr{\tr\omega}\in L^2\Lambda^k(\mathcal{E}_h)$, which are widely used in the DG literature. In particular, for each interior edge $E\in\mathcal{E}_h^o$,
\begin{align*}
    &\db{\tr\omega}|_E:=\frac{1}{2}\left\{\tr_E(\omega|_{T_E^+})+\tr_E(\omega|_{T_E^-})\right\},\quad E\in\mathcal{E}_h^o,\\
    &\lr{\tr\omega}|_E:=\tr_E(\omega|_{T_E^+})-\tr_E(\omega|_{T_E^-}),\quad E\in\mathcal{E}_h^o.
\end{align*}
Furthermore, for each boundary edge $E\in\mathcal{E}_h^\partial$, we distinguish the trace averages/jumps of $\omega$ and $\star\omega$ as
\begin{align*}
    &\db{\tr\omega}|_E:=\tr_E(\omega|_{T_E}),\quad\lr{\tr\omega}|_E:=0,\\
    &\db{\tr\star\omega}|_E:=0,\quad\lr{\tr\star\omega}|_E:=\tr_E(\star\omega|_{T_E}).
\end{align*}
For each $T \in\mathcal{T}_h$ and $E\subset\partial T$, let $s_T^E=1$ if $\nu_E$ is outward to $\partial T$ and $s_T^E=-1$ otherwise. Let $s_T$ be the sign function on $\partial T$ with $s_T=s_T^E$ on $E\subset\partial T$. For $\eta_1, \eta_2\in L^2\Lambda^k(E)$, we define $L^2$-inner products on element boundaries
\begin{align*}
    &\langle\eta_1,\eta_2 \rangle^s_{\partial T}:=\langle s_T\eta_1,\eta_2 \rangle_{\partial T}=\sum_{E\subset\partial T, E\in\mathcal{E}_h}s_T^E\langle\eta_1,\eta_2 \rangle_E,\\
    &\ab{\cdot}{\cdot}^s_{\partial\mathcal{T}}=\sum_{T\in\mathcal{T}}\langle\cdot,\cdot\rangle^s_{\partial T}.
\end{align*}

With this notation and using \eqref{IntegrationByParts}, \eqref{barstar2}, we obtain for $\omega_1\in H^1\Lambda^k(\Th)$, $\omega_2\in H^1\Lambda^{k+1}(\Th)$ and $T\in\mathcal{T}_h,$
\begin{equation}\label{IntegrationByParts2}
(d^k\omega_1,\omega_2)_T=(\omega_1,\delta^{k+1}\omega_2)_T+\ab{\bar{\star}\tr\omega_1}{\tr\star\omega_2}^s_{\partial T}.
\end{equation}
In fact the regularity in \eqref{IntegrationByParts2} could be weakened as $\omega_1|_T\in H\Lambda^k(T)\cap H^s\Lambda^k(T),$ $\omega_2|_T\in L^2\Lambda^{k+1}(T)\cap H^s\Lambda^{k+1}(T),$ $\delta^{k+1}\omega_2|_T\in L^2\Lambda^k(T)$ with $s>\frac{1}{2}$. It is also straightforward to verify the elementary identity
\begin{equation}\label{meanjump}
    \langle \bar{\star}\tr\omega_1,\tr\star\omega_2\rangle_{\partial \mathcal{T}_h}^s=\langle \bar{\star}\db{\tr\omega_1},\lr{\tr\star\omega_2}\rangle+\langle \bar{\star}\lr{\tr\omega_1},\db{\tr\star\omega_2}\rangle.
\end{equation}

Let $d^k_h$ and $\delta^{k+1}_h$ denote the piecewise $d^k$ and $\delta^{k+1}$ w.r.t.~the partition $\Th$, respectively. We define spaces of broken polynomial differential forms
\begin{equation}\label{DGFEpairs}
    \mathcal{V}_h^-=\left\{\begin{aligned}
        &\tilde{\mathcal{P}}_{r_1}\Lambda^{k-1}(\Th),\\
        &\tilde{\mathcal{P}}^-_{r_1}\Lambda^{k-1}(\Th)
    \end{aligned}\right\},\quad \mathcal{V}_h=\left\{\begin{aligned}
        &\tilde{\mathcal{P}}_{r_2}\Lambda^{k}(\Th),\\
        &\tilde{\mathcal{P}}^-_{r_2}\Lambda^{k}(\Th)
    \end{aligned}\right\},\quad \mathcal{V}^+_h=\left\{\begin{aligned}
        &\tilde{\mathcal{P}}_{r_3}\Lambda^{k+1}(\Th),\\
        &\tilde{\mathcal{P}}^-_{r_3}\Lambda^{k+1}(\Th)
    \end{aligned}\right\},
\end{equation}
where $r_1$, $r_2$, $r_3$ are fixed non-negative integers. Given a space $\tilde{V}\subseteq H^1\Lambda^k(\mathcal{T}_h)$, we make use of the trace space \begin{equation*}
    \tr\tilde{V}:=\{v\in L^2\Lambda^k(\mathcal{E}_h): \forall E\in\mathcal{E}_h~ \exists \omega\in \tilde{V}\text{ with }\tr_E\omega=v|_E\}.
\end{equation*}

\subsection{Extended Galerkin Formulation}

\subsubsection{First Order Formulation}
In addition to $\sigma=-\delta^-u$ in Section \ref{SectionPreliminaries}, we introduce another unknown $\xi=-du$ and rewrite \eqref{HodgeLaplace} as
\begin{subequations}\label{1storder}
    \begin{align}
       \sigma+\delta u&=0,\\
       \xi+du&=0,\\
       d^-\sigma+\delta^+\xi&=-f.
    \end{align}
\end{subequations}
A variational problem  of \eqref{1storder} seeks $(\sigma,\xi,u)\in V^-\times W^+\times V$, such that
\begin{equation}\label{mixedHodgeLaplace}
\begin{aligned}
(\sigma,\tau)+(d^-\tau,u)&=0,\quad&&\tau\in V^-,\\
(\xi,\eta)+(du,\eta)&=0,\quad&&\eta\in W^+,\\
(d^-\sigma,v)+(\xi,dv)&=-(f,v),&& v\in V.
\end{aligned}
\end{equation}
The well-posedness of \eqref{mixedHodgeLaplace} will be confirmed by Theorem \ref{continuousinfsup}.

\subsubsection{Derivation of XG methods}
The XG method for \eqref{mixedHodgeLaplace} searches $\sigma_h\in\mathcal{V}_h^-$, $\xi_h\in\mathcal{V}_h^+$,  $u_h\in\mathcal{V}_h$ such that
\begin{equation*}
    \sigma_h\approx\sigma,\quad \xi_h\approx\xi,\quad u_h\approx u.
\end{equation*} 
Since $\Omega$ is a Lipschitz polyhedron, there exists $s_0\in(\frac{1}{2},1]$ relying on $\Omega$ such that 
\begin{equation}\label{regularity}
    u, d^-\sigma, \delta^+\xi\in H^{s_0}\Lambda^k(\Omega),\quad\sigma\in H^{s_0}\Lambda^{k-1}(\Omega),\quad\xi\in H^{s_0}\Lambda^{k+1}(\Omega),
\end{equation}
see, e.g., \cite{ArnoldFalkWinther2006}. Therefore on each element $T$ we could 
test \eqref{1storder} with $H^1$ differential forms and apply the formula \eqref{IntegrationByParts2}. The resulting equations are
\begin{equation}\label{XGcontinuous}
    \begin{aligned}
       &(\sigma,\tau_T)_T+(d^-\tau_T,u)_T-\langle{\bar{\star}\tr\tau_T},\tr\star u\rangle^s_{\partial T}=0,\quad \tau_T\in H^1\Lambda^{k-1}(T),\\
       &(\xi,\eta_T)_T+(\delta^+\eta_T,u)_T+\langle\bar{\star}\tr u,\tr\star\eta_T\rangle^s_{\partial T}=0,\quad\eta_T\in H^1\Lambda^{k+1}(T),\\
       &(\sigma,\delta v_T)_T+(\xi,dv_T)_T+\langle\bar{\star}\tr \sigma,{\tr\star v_T}\rangle^s_{\partial T}\\
       &\qquad-\langle{\bar{\star}\tr v_T},\tr\star \xi\rangle^s_{\partial T}=-(f,v_T)_T,\quad v_T\in H^1\Lambda^k(T).
    \end{aligned}
\end{equation}
To impose weak continuity and facilitate numerical stability, our XG method employs four ``check'' variables $\check{u}_h\in\check{\CV}_h$, $\check{u}^\star_h\in \check{\CV}^\star_h$, $\check{\sigma}_h\in\check{\CV}^-_h$, $\check{\xi}^\star_h\in\check{\CV}^{+\star}_h$, where  
\begin{align*}
    &\check{\CV}^-_h=\mathcal{P}_{r_4}\Lambda^{k-1}(\mathcal{E}_h),\quad\check{\CV}_h=\mathcal{P}_{r_5}\Lambda^k(\mathcal{E}_h),\\
    &\check{\CV}^{+\star}_h=\{\check{\xi}_h^\star\in\mathcal{P}_{r_6}\Lambda^{n-k-1}(\mathcal{E}_h): \check{\xi}_h^\star=0\text{ on }\partial\Omega\},\\
    &\check{\CV}^\star_h=\{\check{v}_h^\star\in\mathcal{P}_{r_7}\Lambda^{n-k}(\mathcal{E}_h): \check{v}_h^\star=0\text{ on }\partial\Omega\},
\end{align*}
with $\{r_i\}_{i=4}^7$ being fixed non-negative integers. We choose numerical fluxes
\begin{align*}
  &\hat{u}_h\in L^2\Lambda^k(\mathcal{E}_h),\quad\hat{u}^\star_h\in L^2\Lambda^{n-k}(\mathcal{E}_h),\\
  &\hat{\sigma}_h\in L^2\Lambda^{k-1}(\mathcal{E}_h),\quad \hat{\xi}^\star_h\in L^2\Lambda^{n-k-1}(\mathcal{E}_h),
\end{align*}
such that on each face in $\mathcal{E}_h$,  \begin{equation*}
    \hat{u}_h\approx\tr u,\quad \hat{u}^\star_h\approx\tr\star u,\quad \hat{\sigma}_h\approx\tr\sigma,\quad\hat{\xi}^\star_h\approx\tr\star\xi.
\end{equation*}
Let $\bar{u}_h$, $\bar{u}^\star_h$, $\bar{\sigma}_h$, $\bar{\xi}^\star_h$ be central fluxes
\begin{equation}\label{bar}
  \begin{aligned}
    &\bar{u}_h:=\db{\tr  u_h},\quad\bar{u}_h^\star:=\db{\tr\star u_h},\\
    &\bar{\sigma}_h:=\db{\tr \sigma_h},\quad\bar{\xi}_h^\star:=\db{\tr\star\xi_h}.
  \end{aligned}
\end{equation}
The numerical fluxes of the XG method have the following distinctive feature
\begin{equation}\label{hat}
\begin{aligned}
    &\hat{u}_h=\bar{u}_h+\check{u}_h\text{ on }\mathcal{E}_h,\quad\hat{\sigma}_h=\bar{\sigma}_h+\check{\sigma}_h\text{ on }\mathcal{E}_h,\\
    &\hat{u}^\star_h=\bar{u}^\star_h+\check{u}^\star_h\text{ on }\mathcal{E}^o_h,\quad\hat{\xi}^\star_h=\bar{\xi}^\star_h+\check{\xi}^\star_h\text{ on }\mathcal{E}^o_h,\\
    &\hat{u}^\star_h=\check{u}^\star_h=0\text{ on }\mathcal{E}^\partial_h,\quad\hat{\xi}^\star_h=\check{\xi}^\star_h=0\text{ on }\mathcal{E}^\partial_h,
\end{aligned}
\end{equation}

In view of the local equation \eqref{XGcontinuous}, we propose a seven-field XG formulation for \eqref{HodgeLaplace}: Find $(\sigma_h,\check{\sigma}_h,\xi_h,\check{\xi}^\star_h,u_h,\check{u}_h,\check{u}^\star_h)\in \CV_h^-\times\check{\CV}_h^-\times\CV_h^+\times\check{\CV}_h^{+\star}\times\CV_h\times\check{\CV}_h\times\check{\CV}^\star_h$, such that
on each $T\in\mathcal{T}_h,$ 
\begin{subequations}\label{XGlocal}
    \begin{align}
       &(\sigma_h,\tau_h)_T+(d^-\tau_h,u_h)_T-\langle{\bar{\star}\tr\tau_h},\hat{u}_h^\star\rangle^s_{\partial T}=0,\label{XGT1}\\
       &(\xi_h,\eta_h)_T+(\delta^+\eta_h,u_h)_T+\langle\bar{\star}\hat{u}_h,\tr\star\eta_h\rangle^s_{\partial T}=0,\label{XGT2}\\
       &(\sigma_h,\delta v_h)_T+(\xi_h,dv_h)_T+\langle\bar{\star}\hat{\sigma}_h,{\tr\star v_h}\rangle^s_{\partial T}-\langle{\bar{\star}\tr v_h},\hat{\xi}_h^\star\rangle^s_{\partial T}=-(f,v_h)_T,\label{XGT3}\\
    &\ab{\bar{\star}\check{\sigma}_h-\textsf{a}\lr{\tr\star u_h}}{\bar{\star}\check{\tau}_h}=0,\label{checka}\\
    &\ab{\check{\xi}_h^\star-\textsf{b}\bar{\star}\lr{\tr u_h}}{\check{\eta}_h^\star}=0,\label{checkb}\\
    &\ab{\bar{\star}\check{u}_h-\textsf{c}\lr{\tr \star\xi_h}}{\bar{\star}\check{v}_h}=0,\label{checkc}\\
    &\ab{\check{u}_h^\star-\textsf{d}\bar{\star}\lr{\tr\sigma_h}}{\check{v}_h^\star}=0,\label{checkd}
  \end{align}
\end{subequations}
for all $(\tau_h,\check{\tau}_h,\eta_h,\check{\eta}^\star_h,v_h,\check{v}_h,\check{v}_h^\star)\in \CV_h^-\times\check{\CV}_h^-\times\CV_h^+\times\check{\CV}_h^{+\star}\times\CV_h\times\check{\CV}_h\times\check{\CV}_h^\star$.
Here $\textsf{a}, \textsf{b}, \textsf{c}, \textsf{d}\in\mathcal{P}_0(\mathcal{E}_h)=\mathcal{P}_0\Lambda^0(\mathcal{E}_h)$ are piecewise constants that will be specified in the next section. 
\begin{remark}\label{consistency}
It follows from \eqref{XGcontinuous} that \eqref{XGT1}--\eqref{XGT3} hold with $(\sigma,\xi,u)$ replacing $(\sigma_h,\xi_h,u_h)$. In addition, \eqref{checka}--\eqref{checkd} are true provided $\check{\sigma}_h=0$, $\check{\xi}^\star_h=0$, $\check{u}_h=0$, $\check{u}^\star_h=0$ and $(\sigma_h,\xi_h,u_h)$ is replaced by $(\sigma,\xi,u)$. Therefore, the overall XG method is consistent.
\end{remark}

Equations \eqref{checka}--\eqref{checkd} are motivated by Nitsche's trick. First it follows from $\sigma\in H\Lambda^{k-1}(\Omega)$, $u\in H\Lambda^k(\Omega)$ that $\lr{\tr\sigma}=0$, $\lr{\tr u}=0$ on $\mathcal{E}_h^o$. In the XG method, such inter-element continuity is weakly enforced by  \eqref{checkb}, \eqref{checkd}. On the other hand, $u\in \mathring{H}^*\Lambda^k(\Omega)$, $\xi\in \mathring{H}^*\Lambda^{k+1}(\Omega)$ imply that $\lr{\tr\star u}=0$, $\lr{\tr\star\xi}=0$ on $\mathcal{E}_h$. Weak inter-element continuity from this viewpoint could be imposed by \eqref{checka}, \eqref{checkc}. Therefore, the XG method \eqref{XGlocal} is a compromise between the following two types of conformity
\begin{subequations}\label{conformity}
\begin{align}
    &(\sigma,u)\in H\Lambda^{k-1}(\Omega)\times H\Lambda^k(\Omega),\label{conformity1}\\
    &(\xi,u)\in \mathring{H}^*\Lambda^{k+1}(\Omega)\times \mathring{H}^*\Lambda^k(\Omega).\label{conformity2}
\end{align}
\end{subequations}

For simplicity, we make the assumption
\begin{equation}\label{CheckAssumption}
\begin{aligned}
       &\tr\star\mathcal{V}_h\subseteq\bar{\star}\check{\mathcal{V}}_h^-,\quad\bar{\star}\tr\mathcal{V}_h\subseteq\check{\mathcal{V}}_h^{+\star},\\
       &\tr\star\mathcal{V}^+_h\subseteq\bar{\star}\check{\mathcal{V}}_h,\quad\bar{\star}\tr\mathcal{V}^-_h\subseteq\check{\mathcal{V}}_h^\star.
\end{aligned}
\end{equation}
It then follows that all check variables can be eliminated using the relations
\begin{equation}\label{check}
    \begin{aligned}
 &\bar{\star}\check{\sigma}_h=\textsf{a}\lr{\tr\star u_h},\quad\check{\xi}_h^\star=\textsf{b}\bar{\star}\lr{\tr u_h},\\
    &\bar{\star}\check{u}_h=\textsf{c}\lr{\tr \star\xi_h},\quad\check{u}_h^\star=\textsf{d}\bar{\star}\lr{\tr\sigma_h},
    \end{aligned}
\end{equation}
see \eqref{XGcompact} for a compact three-field formulation. 
However, with the help of several check variables, we are able to present a more transparent convergence analysis for the original augmented scheme.

\textbf{Four Field Formulation I.}~Applying  \eqref{IntegrationByParts2} to \eqref{XGT1}--\eqref{XGT3} leads to
\begin{equation}\label{XGlocal2}
    \begin{aligned}
       &(\sigma_h,\tau_h)_T+(d^-\tau_h,u_h)_T-\langle{\bar{\star}\tr\tau_h},\hat{u}_h^\star\rangle^s_{\partial T}=0,\\
       &(\xi_h,\eta_h)_T+(\eta_h,du_h)_T+\langle \bar{\star}(\hat{u}_h-\tr u_h),\tr\star\eta_h\rangle^s_{\partial T}=0,\\
       &(d^-\sigma_h,v_h)_T+(\xi_h,dv_h)_T+\langle \bar{\star}(\hat{\sigma}_h-\tr\sigma_h),{\tr\star v_h}\rangle^s_{\partial T}\\
       &\qquad-\langle{\bar{\star}\tr v_h},\hat{\xi}_h^\star\rangle^s_{\partial T}=-(f,v_h)_T.
    \end{aligned}
\end{equation}
Summing \eqref{XGlocal2} over all $T\in\mathcal{T}_h$ and using \eqref{hat}, \eqref{bar}, \eqref{meanjump}, we obtain
\begin{equation}\label{XGglobal}
    \begin{aligned}
       &(\sigma_h,\tau_h)+(d^-_h\tau_h,u_h)-\langle{\bar{\star}\lr{\tr\tau_h}},\db{\tr\star{u}_h}\rangle-\langle{\bar{\star}\lr{\tr\tau_h}},\check{u}_h^\star\rangle=0,\\
       &(\xi_h,\eta_h)+(\eta_h,d_hu_h)+\langle\bar{\star}\check{u}_h,\lr{\tr\star\eta_h}\rangle-\langle\bar{\star}\lr{\tr u_h},\db{\tr\star\eta_h}\rangle=0,\\
       &(d_h^-\sigma_h,v_h)+(\xi_h,d_hv_h)+\langle\bar{\star}\check{\sigma}_h,\lr{\tr\star v_h}\rangle-\langle\bar{\star}\lr{\tr\sigma_h},\db{\tr\star v_h}\rangle\\
       &\qquad-\langle{\bar{\star}\lr{\tr v_h}},\db{\tr\star\xi_h}\rangle-\langle{\bar{\star}\lr{\tr v_h}},\check{\xi}_h^\star\rangle=-(f,v_h).
    \end{aligned}
\end{equation}
Eliminating $\check{\sigma}_h, \check{u}_h, \check{u}^\star_h$ by \eqref{check}, the XG method \eqref{XGlocal} is equivalent to the following four-field formulation:  Find $\sigma_h\in\mathcal{V}_h^-$, $\xi_h\in\mathcal{V}_h^+$, $\check{\xi}_h^\star\in\check{\mathcal{V}}_h^{+\star}$, $u_h\in\mathcal{V}_h$,  such that
\begin{equation}\label{XGreduced0}
    \begin{aligned}
       &(\sigma_h,\tau_h)+(d^-_h\tau_h,u_h)-\langle{\bar{\star}\lr{\tr\tau_h}},\db{\tr\star{u}_h}\rangle-\langle\textsf{d}\lr{\tr\sigma_h},\lr{\tr\tau_h}\rangle=0,\\
       &(\xi_h,\eta_h)+(\eta_h,d_hu_h)+\langle\textsf{c}\lr{\tr\star\xi_h},\lr{\tr\star\eta_h}\rangle-\langle\bar{\star}\lr{\tr u_h},\db{\tr\star\eta_h}\rangle=0,\\
       &(d_h^-\sigma_h,v_h)+(\xi_h,d_hv_h)+\langle\textsf{a}\lr{\tr\star u_h},\lr{\tr\star v_h}\rangle-\langle\bar{\star}\lr{\tr\sigma_h},\db{\tr\star v_h}\rangle\\
       &\quad-\langle{\bar{\star}\lr{\tr v_h}},\db{\tr\star\xi_h}\rangle-\langle{\bar{\star}\lr{\tr v_h}},\check{\xi}_h^\star\rangle=-(f,v_h),\\
       &\ab{\check{\xi}_h^\star-\textsf{b}\bar{\star}\lr{\tr u_h}}{\check{\eta}_h^\star}=0,
    \end{aligned}
\end{equation}
for all $(\tau_h,\eta_h,\check{\eta}_h^\star,v_h)\in\mathcal{V}_h^-\times\mathcal{V}_h^+\times\check{\mathcal{V}}_h^{+\star}\times\mathcal{V}_h$.

\textbf{Four Field Formulation II.}
Similar arguments in \eqref{XGlocal2}, \eqref{XGglobal} yield
\begin{equation}
    \begin{aligned}
       &(\sigma_h,\tau_h)+(\delta_hu_h,\tau_h)+\langle{\bar{\star}\db{\tr\tau_h}},\lr{\tr\star u_h}\rangle-\langle{\bar{\star}\lr{\tr\tau_h}},\check{u}_h^\star\rangle=0,\\
       &(\xi_h,\eta_h)+(\delta^+_h\eta_h,u_h)+\langle\bar{\star}\db{\tr u_h},\lr{\tr\star\eta_h}\rangle+\langle\bar{\star}\check{u}_h,\lr{\tr\star\eta_h}\rangle=0,\\
       &(\sigma_h,\delta_hv_h)+(\delta_h^+\xi_h,v_h)+\langle\bar{\star}\db{\tr\sigma_h},\lr{\tr\star v_h}\rangle+\langle\bar{\star}\check{\sigma}_h,\lr{\tr\star v_h}\rangle\\
       &\quad-\langle{\bar{\star}\lr{\tr v_h}},\check{\xi}_h^\star\rangle+\langle{\bar{\star}\db{\tr v_h}},\lr{\tr\star\xi_h}\rangle=-(f,v_h).
  \end{aligned}
\end{equation}
Eliminating $\check{\xi}^\star_h, \check{u}_h, \check{u}^\star_h$ by \eqref{check}, the XG method \eqref{XGlocal} is also equivalent to another four-field formulation:  Find $\sigma_h\in\mathcal{V}_h^-$, $\check{\sigma}_h\in\check{\mathcal{V}}_h^-$, $\xi_h\in\mathcal{V}_h^+$,  $u_h\in\mathcal{V}_h$,  such that
\begin{equation}\label{XGreducedII}
    \begin{aligned}
       &(\sigma_h,\tau_h)+(\delta_hu_h,\tau_h)+\langle{\bar{\star}\db{\tr\tau_h}},\lr{\tr\star u_h}\rangle-\langle\textsf{d}\lr{\tr\sigma_h},\lr{\tr\tau_h}\rangle=0,\\
       &(\xi_h,\eta_h)+(\delta^+_h\eta_h,u_h)+\langle\bar{\star}\db{\tr u_h},\lr{\tr\star\eta_h}\rangle+\langle\textsf{c}\lr{\tr\star\xi_h},\lr{\tr\star\eta_h}\rangle=0,\\
       &(\sigma_h,\delta_hv_h)+(\delta_h^+\xi_h,v_h)+\langle\bar{\star}\db{\tr\sigma_h},\lr{\tr\star v_h}\rangle+\langle\bar{\star}\check{\sigma}_h,\lr{\tr\star v_h}\rangle\\
       &\quad-\langle\textsf{b}\lr{\tr u_h},\lr{\tr v_h}\rangle+\langle{\bar{\star}\db{\tr v_h}},\lr{\tr\star\xi_h}\rangle=-(f,v_h),\\
       &\langle\bar{\star}\check{\sigma}_h-\textsf{a}\lr{\tr\star u_h},\bar{\star}\check{\tau}_h\rangle=0,
  \end{aligned}
\end{equation}
for all $(\sigma_h,\check{\sigma}_h,\eta_h,v_h)\in\mathcal{V}_h^-\times\check{\mathcal{V}}_h^-\times\mathcal{V}_h^+\times\mathcal{V}_h$.

\textbf{Three Field Formulation.}~Eliminating $\check{\xi}_h^\star$ by \eqref{checkb}, the scheme \eqref{XGreduced} further yields a three field formulation: Find $\sigma_h\in\mathcal{V}_h^-$, $\xi_h\in\mathcal{V}_h^+$,$u_h\in\mathcal{V}_h$ such that
\begin{equation}\label{XGcompact}
    \begin{aligned}
       &(\sigma_h,\tau_h)+(d^-_h\tau_h,u_h)-\langle{\bar{\star}\lr{\tr\tau_h}},\db{\tr\star{u}_h}\rangle-\langle{\textsf{d}\lr{\tr\sigma_h}},\lr{\tr\tau_h}\rangle=0,\\
       &(\xi_h,\eta_h)+(\eta_h,d_hu_h)+\langle\textsf{c}\lr{\tr \star\xi_h},\lr{\tr\star\eta_h}\rangle-\langle\bar{\star}\lr{\tr u_h},\db{\tr\star\eta_h}\rangle=0,\\
       &(d_h^-\sigma_h,v_h)+(\xi_h,d_hv_h)+\langle\textsf{a}\lr{\tr\star u_h},\lr{\tr\star v_h}\rangle-\langle\bar{\star}\lr{\tr\sigma_h},\db{\tr\star v_h}\rangle\\
       &-\langle{\bar{\star}\lr{\tr v_h}},\db{\tr\star\xi_h}\rangle-\langle{\textsf{b}\lr{\tr u_h}},\lr{\tr v_h}\rangle=-(f,v_h),
    \end{aligned}
\end{equation}
for all $(\tau_h,\eta_h,v_h)\in\mathcal{V}_h^-\times\mathcal{V}_h^+\times\mathcal{V}_h$. 

Throughout the rest of this paper, 
we use $C$ to denote a generic constant depending only on $\Omega$, $C_{\text{mesh}}$, and the polynomial degrees $\{r_i\}_{i=1}^7$. We may also use $C_t>0$ to indicate an absolute constant relying the previous quantities and the extra parameter $t$. We say $A\lesssim B$ if $A\leq CB$, and $A\simeq B$ provided $A\lesssim B$ and $B\lesssim A.$

\section{Inf-sup conditions}\label{SectionInfsup}
In this section, we analyze the inf-sup stability of the continuous mixed formulation \eqref{mixedHodgeLaplace} and the four-field XG formulation \eqref{XGreduced0}. 

\subsection{Continuous Inf-Sup Condition}
First we introduce  
\begin{align*}
    a(\sigma,\xi;\tau,\eta)&:=(\sigma,\tau)+(\xi,\eta),\\
    b(\sigma,\xi;v)&:=(d^-\sigma,v)+(\xi,dv),
\end{align*} 
and the total bilinear form
\begin{align*}
  B(\sigma,\xi,u;\tau,\eta,v)=a(\sigma,\xi;\tau,\eta)+b(\sigma,\xi;v)+b(\tau,\eta;u).
\end{align*}
The continuous problem \eqref{mixedHodgeLaplace} is then written as 
\begin{equation*}
    B(\sigma,\xi,u;\tau,\eta,v)=-(f,v),\quad\forall (\tau,\eta,v)\in V^-\times W^+\times V.
\end{equation*}
The next theorem confirms the well-posedness of  \eqref{mixedHodgeLaplace}.
\begin{theorem}[Continuous Inf-Sup Condition]\label{continuousinfsup}
There exists a constant $C_{\emph{cts}}$ dependent on $\Omega$ such that
\begin{equation*}
    \inf_{\substack{(\sigma,\xi,u)\in {V^-}\times W^+\times V\\ \|\sigma\|_{V^-}^2+\|\xi\|^2+\|u\|^2_V=1}}\sup_{\substack{(\tau,\eta,v)\in V^-\times W^+\times V\\ \|\tau\|_{V^-}^2+\|\eta\|^2+\|v\|^2_V=1}} B(\sigma,\xi,u;\tau,\eta,v)\geq C_{\emph{cts}}>0.
\end{equation*}
\end{theorem}
\begin{proof}
Consider the kernel of $b$
\begin{equation*}
    N:=\{(\tau,\eta)\in V^-\times W^+: b(\tau,\eta;v)=0,~\forall v\in V\}.
\end{equation*}
Given $(\tau,\eta)\in N$, taking $v=d^-\tau$ in $b(\tau,\eta;v)$ leads to $d^-\tau=0$. It then follows that
\begin{equation}\label{SPDa}
 a(\tau,\eta;\tau,\eta)=\|\tau\|_{V^-}^2+\|\eta\|^2.  
\end{equation}

Next we verify the inf-sup condition of $b$.
For any $v\in V,$ using \eqref{Helm} and \eqref{Poincare}, we obtain 
\begin{subequations}\label{vdecomp}
\begin{align}
& v=v_\mathfrak{B}+v_\perp,\quad v_\mathfrak{B}\in\mathfrak{B},~v_\perp\in\mathfrak{Z}^{\perp},\\
&\|v_\perp\|\leq c_P\|dv\|.
\end{align}
\end{subequations}
Moreover, we simply take
\begin{equation}\label{eta}
  \eta=dv\in W^+,
  \end{equation}
and choose $\tau\in V^-$ such that \begin{equation}\label{tau}
  d^-\tau=u_\mathfrak{B},\quad \|\tau\|_V\leq c_P\|u_\mathfrak{B}\|.
  \end{equation}
A combination of \eqref{vdecomp}, \eqref{eta}, \eqref{tau}, $(u_\mathfrak{B},u_\perp)=0$ yields
  \begin{align*}
    b(\tau,\eta;v)&=(d^-\tau,v)+(\eta,dv)\\
    &=\|v_{\mathfrak{B}}\|^2+\|dv\|^2\\
    &\geq 2^{-1}\min(1,c_P^{-2})\|v\|_V^2.
  \end{align*}
It then follows that
\begin{equation}\label{infsupb}
\inf_{\substack{v\in V\\ \|v\|_V=1}}\sup_{\substack{(\tau,\eta)\in V^-\times W^+\\ \|\tau\|_{V^-}^2+\|\eta\|^2=1}} b(\tau,\eta;v)\geq C_b,
\end{equation}
where $C_b$ depends only on $c_P.$

Finally the inf-sup condition of $B$ follows from \eqref{SPDa} and \eqref{infsupb}, see, e.g., \cite{Brezzi1974,XuZikatanov2003} for the theory of inf-sup condition for saddle point systems.
\end{proof}

\subsection{Discrete Inf-Sup Condition}
In this subsection, we analyze the discrete inf-sup condition of the reduced four-field XG method \eqref{XGreduced0} based on appropriate discrete spaces. First of all, we make the assumption
\begin{equation}\label{Assumption2}
\mathcal{V}_h^-\cap H\Lambda^{k-1}(\Omega)=V_h^-,\quad\mathcal{V}_h\cap H\Lambda^{k}(\Omega)=V_h.
\end{equation} 
For convenience, several tuples are introduced
\begin{align*}
&\Sigma_h=\CV^-_h\times\CV_h^+\times\check{\CV}_h^{+\star},\\
    &\tilde{\sigma}_h=(\sigma_h,\xi_h,\check{\xi}^\star_h)\in\Sigma_h,\\ &\tilde{\tau}_h=(\tau_h,\eta_h,\check{\eta}^\star_h)\in\Sigma_h.
\end{align*}
We also make use of the discrete bilinear forms
\begin{align*}
       a_h(\tilde{\sigma}_h,\tilde{\tau}_h)&=(\sigma_h,\tau_h)+(\xi_h,\eta_h)-\langle{\textsf{d}\lr{\tr\sigma_h}},\lr{\tr\tau_h}\rangle\\
       &\quad+\langle\textsf{c}\lr{\tr\star\xi_h},\lr{\tr\star\eta_h}\rangle+\ab{\textsf{b}^{-1}\check{\xi}^\star_h}{\check{\eta}^\star_h},\\
       b_h(\tilde{\sigma}_h,v_h)&=(d_h^-\sigma_h,v_h)+(\xi_h,d_hv_h)-\langle\bar{\star}\lr{\tr\sigma_h},\db{\tr\star v_h}\rangle\\
       &\quad-\langle{\bar{\star}\lr{\tr v_h}},\db{\tr\star\xi_h}\rangle-\langle\bar{\star}\lr{\tr v_h},\check{\xi}_h^\star\rangle,\\
       c_h(u_h,v_h)&=\langle\textsf{a}\lr{\tr\star u_h},\lr{\tr\star v_h}\rangle.
    \end{align*}
With the above notation, \eqref{XGreduced0} is to find $(\tilde{\sigma}_h,u_h)\in\Sigma_h\times\mathcal{V}_h$ such that
\begin{equation}\label{XGreduced}
    \begin{aligned}
    a_h(\tilde{\sigma}_h,\tilde{\tau}_h)+b_h(\tilde{\tau}_h,u_h)&=0,\\
    b_h(\tilde{\sigma}_h,v_h)+c_h(u_h,v_h)&=-(f,v_h)
    \end{aligned}
\end{equation}
for all $(\tilde{\tau}_h,v_h)\in\Sigma_h\times\mathcal{V}_h.$     
    
Let $h$ be the mesh size function such that 
\begin{align*}
    &h|_T=h_T:=\text{diam}(T),\quad T\in\mathcal{T}_h,\\
    &h|_E=h_E:=\text{diam}(E),\quad E\in\mathcal{E}_h.
\end{align*}
Given a positive constant $\rho$ and $\tau_h\in\mathcal{V}^-_h$, $v_h\in\mathcal{V}_h$, $\tilde{\tau}_h\in\Sigma_h$, we propose the following mesh dependent norms
\begin{align*}
           \vertiii{\tau_h}_{\rho,h}^2&:=\|\tau_h\|^2+\|d^-_h\tau_h\|^2+\rho^{-1}\|h^{-\frac{1}{2}}\lr{\tr \tau_h}\|_{\mathcal{E}^o_h}^2,\\ \vertiii{v_h}_{\rho,h}^2&:=\|v_h\|^2+\|d_hv_h\|^2+\rho^{-1}\|h^{-\frac{1}{2}}\lr{\tr v_h}\|_{\mathcal{E}^o_h}^2,\\
    \vertiii{\tilde{\tau}_h}^2_{\rho ,h}&:=\|\tau_h\|_{\rho,h}+\|\eta_h\|^2+\rho\|h^\frac{1}{2}\check{\eta}^\star_h\|_{\mathcal{E}^o_h}^2,\\
    \vertiii{(\tilde{\tau}_h,v_h)}^2_{\rho ,h}&:=\|\tilde{\tau}_h\|^2_{\rho ,h}+\|v_h\|_{\rho,h}^2.
\end{align*}

The next lemma shows that the DG spaces of differential forms could be approximated by conforming AFW spaces. The proof is postponed in the appendix.  
\begin{lemma}\label{approximation}
For any $v_h\in\mathcal{V}_h$ and $\tau_h\in\mathcal{V}^-_h$, there exist $v_h^c\in\mathcal{V}_h$  and $\tau_h^c\in\mathcal{V}^-_h$ such that
\begin{subequations}
\begin{align}
  &\|h^{-1}(\tau_h-\tau_h^c)\|+\|d^-_h(\tau_h-\tau_h^c)\|\leq C\|h^{-\frac{1}{2}}\lr{\tr \tau_h}\|_{\mathcal{E}^o_h},\label{approx1}\\
  &\|h^{-1}(v_h-v_h^c)\|+\|d_h(v_h-v_h^c)\|\leq C\|h^{-\frac{1}{2}}\lr{\tr v_h}\|_{\mathcal{E}^o_h}.\label{approx2}
\end{align}
\end{subequations}
\end{lemma}

In addition, we will frequently utilize the estimate
\begin{equation}\label{trace}
    \|\tr v_h\|_{\partial T}\leq C_r h_T^{-\frac{1}{2}}\|v_h\|_T,\quad\forall v_h\in\mathcal{P}_r\Lambda^k(T),\quad\forall T\in\mathcal{T}_h.
\end{equation}
In fact, \eqref{trace} immediately follows from
the well-known 
trace and inverse inequalities. The next lemma is concerned with the inf-sup condition of the bilinear form $b_h$.
\begin{lemma}[Inf-Sup Condition of $b_h$]\label{lemmabh}
Let $\rho_0>0$ be a constant and $\rho\in(0,\rho_0].$  
Let the assumptions \eqref{CheckAssumption}, \eqref{Assumption2} hold and assume $d_h\mathcal{V}_h\subseteq\mathcal{V}_h^+$.
Then we have
\begin{equation*}
    \inf_{\substack{v_h\in\CV_h\\v_h\neq0}}\sup_{\substack{\tilde{\tau}_h\in\Sigma_h\\\tilde{\tau}_h\neq0}}\frac{b_h(\tilde{\tau}_h,v_h)}{\vertiii{\tilde{\tau}_h}_{\rho,h}\vertiii{v_h}_{\rho,h}}=C_{\rho_0}>0,
\end{equation*}
\end{lemma}
\begin{proof}
Given $v_h\in \mathcal{V}_h,$ we choose $v^c_h\in V_h$  in Lemma \ref{approximation} such that
\begin{equation}\label{approxv}
    \|v_h-v_h^c\|+\|d_h(v_h-v_h^c)\|\leq C_{\rho_0}\rho^{-\frac{1}{2}}\|h^{-\frac{1}{2}}\lr{\tr v_h}\|_{\mathcal{E}^o_h}.
\end{equation}
Combining \eqref{approxv} and the triangle inequality, we obtain
\begin{equation}\label{bdvhc}
    \|v_h^c\|_V\leq C_{\rho_0}\vertiii{v_h}_{\rho,h}.
\end{equation}
Due to \eqref{discreteHodge} and \eqref{discretePoincare}, there exists a stable decomposition
\begin{subequations}\label{vhc}
\begin{align}
    &v^c_h=d^-\tau_h+v_h^\perp,\quad \tau_h\in V^-_h,~ v_h^\perp\in\mathfrak{Z}_h^\perp,\\
    &\|v_h^\perp\|\leq c_{dP}\|dv_h^\perp\|=c_{dP}\|dv_h^c\|,\label{dperp}\\
    &\|\tau_h\|_{V^-}\lesssim\|d^-\tau_h\|\leq\|v_h^c\|.\label{dtau}
    \end{align}
\end{subequations}
It then follows from \eqref{approxv}, \eqref{vhc} and the triangle inequalities that
\begin{equation}\label{vhbound}
\begin{aligned}
    \vertiii{v_h}_{\rho,h}&\leq\vertiii{v_h^c}_{\rho,h}+\vertiii{v_h-v_h^c}_{\rho,h}\\
    &\leq \|v_h^c\|+\|dv_h^c\|+C_{\rho_0}\rho^{-\frac{1}{2}}\|h^{-\frac{1}{2}}\lr{\tr v_h}\|_{\mathcal{E}^o_h}\\
    &\leq C_{\rho_0}\big(\|d^-\tau_h\|+\|d_hv_h\|+\rho^{-\frac{1}{2}}\|h^{-\frac{1}{2}}\lr{\tr v_h}\|_{\mathcal{E}^o_h}\big).
\end{aligned}
\end{equation}
We now take $\tilde{\tau}_h=(\tau_h,\eta_h,\check{\eta}_h^\star)\in\Sigma_h$, where $\eta_h\in\mathcal{V}_h^+$, $\check{\eta}_h^\star\in\check{\mathcal{V}}_h^{+\star}$ are defined as 
\begin{equation}\label{etah}
    \eta_h=d_hv_h,\quad \check{\eta}_h^\star=-t\rho^{-1}h^{-1}\bar{\star}\lr{\tr v_h}
\end{equation} with some constant $t$ to be specified later. Direct calculation shows 
\begin{align*}
    b_h(\tilde{\tau}_h,v_h)&=(d^-\tau_h,v_h^c)+(d^-\tau_h,v_h-v_h^c)+\|d_hv_h\|^2\\
    &-\langle{\bar{\star}\lr{\tr v_h}},\db{\tr\star d_hv_h}\rangle+t\rho^{-1}\|h^{-\frac{1}{2}}\lr{\tr v_h}\|_{\mathcal{E}^o_h}^2.
\end{align*}
Using the previous identity, the estimates \eqref{dtau}, \eqref{approxv}, \eqref{trace}, and a Young's inequality with  arbitrary $\varepsilon>0$, we obtain that
\begin{align*}
    &b_h(\tilde{\tau}_h,v_h)\geq\|d^-\tau_h\|^2-C\|d^-\tau_h\|\|h^{-\frac{1}{2}}\lr{\tr v_h}\|_{\mathcal{E}^o_h}+\|d_hv_h\|^2\\
    &\qquad-C\|h^{-\frac{1}{2}}\lr{\tr v_h}\|_{\mathcal{E}^o_h}\|d_hv_h\|+t\rho^{-1}\|h^{-\frac{1}{2}}\lr{\tr v_h}\|_{\mathcal{E}^o_h}^2\\
    &\quad\geq (1-2^{-1}\rho\varepsilon C^2)\big(\|d^-\tau_h\|^2+\|d_hv_h\|^2\big)+(t-\varepsilon^{-1})\rho^{-1}\|h^{-\frac{1}{2}}\lr{\tr v_h}\|_{\mathcal{E}^o_h}^2.
\end{align*}
In the above inequality, taking $\varepsilon=\rho^{-1}C^{-2}$, $t=2\rho_0 C^2$ and using \eqref{vhbound} lead to
\begin{equation}\label{blower}
\begin{aligned}
    &b_h(\tilde{\tau}_h,v_h)\\
    &\geq2^{-1}\big(\|d^-\tau_h\|^2+\|d_hv_h\|^2\big)+C_{\rho_0}\rho^{-1}\|h^{-\frac{1}{2}}\lr{\tr v_h}\|_{\mathcal{E}_h^o}^2\\
    &\geq C_{\rho_0}\vertiii{v_h}^2_{\rho,h}.
    \end{aligned}
\end{equation}
On the other hand, a combination of \eqref{dtau}, \eqref{etah}, \eqref{bdvhc} yields
\begin{equation}\label{tauetaupper}
    \vertiii{\tilde{\tau}_h}_{\rho,h}\leq C_{\rho_0}\vertiii{v_h}_{\rho,h}.
\end{equation}
Finally we finish the proof by \eqref{blower} and \eqref{tauetaupper}.
\end{proof}

Now it remains to verify that $a_h$ is coercive on the kernel of $b_h$.
\begin{lemma}[Coercivity of $a_h$]\label{lemmaah}
Let $\rho_0>0$ be a constant and $\rho\in(0,\rho_0].$ Let \begin{align*}
    &\textsf{b}\simeq\rho ^{-1}h^{-1},\quad \textsf{c}\simeq\rho h,\quad \textsf{d}\simeq-\rho ^{-1}h^{-1},\\
    &\mathcal{N}_h=\{\tilde{\tau}_h\in\Sigma_h: b_h(\tilde{\tau}_h,v_h)=0,~\forall v_h\in\mathcal{V}_h\}.
\end{align*} 
Then it holds that 
\begin{equation*}
    {a_h(\tilde{\tau}_h,\tilde{\tau}_h)}\geq C_{\rho_0}\vertiii{\tilde{\tau}_h}^2_{\rho,h},\quad\forall\tilde{\tau}_h\in\mathcal{N}_h.
\end{equation*}
\end{lemma}
\begin{proof}
Given $\tilde{\tau}_h=(\tau_h,\eta_h,\check{\eta}^\star_h)\in\mathcal{N}_h$, the definition of $a_h$ implies
\begin{equation}\label{ahlower}
    a_h(\tilde{\tau}_h,\tilde{\tau}_h)\gtrsim\|\tau_h\|^2+\|\eta_h\|^2+\rho^{-1}\|h^{-\frac{1}{2}}\lr{\tr\tau_h}\|^2_{\mathcal{E}^o_h}+\rho\|h^\frac{1}{2}\check{\eta}^\star_h\|_{\mathcal{E}_h^o}^2.
\end{equation}
For any $v_h\in\mathcal{V}_h$,
the membership of $\mathcal{N}_h$ yields
\begin{equation}\label{Nh}
    \begin{aligned}
&(d_h^-\tau_h,v_h)+(\eta_h,d_hv_h)-\langle\bar{\star}\lr{\tr\tau_h},\db{\tr\star v_h}\rangle\\
       &\quad-\langle\bar{\star}\lr{\tr v_h},\db{\tr\star\eta_h}\rangle-\langle\bar{\star}\lr{\tr v_h},\check{\eta}^\star_h\rangle=0.
\end{aligned}
\end{equation}
We take $\tau_h^c\in V^-_h$ in Lemma \ref{approximation} such that 
\begin{equation}\label{approxtau}
    \|\tau_h-\tau_h^c\|+\|d_h(\tau_h-\tau_h^c)\|\leq C_{\rho_0}\rho^{-\frac{1}{2}}\|h^{-\frac{1}{2}}\lr{\tr\tau_h}\|_{\mathcal{E}^o_h}.
\end{equation}
It follows from \eqref{Nh} with $v_h=d^-\tau^c_h$ and \eqref{approxtau}, \eqref{trace} that
\begin{align*}
    &(d_h^-\tau_h,d^-\tau^c_h)=\ab{\bar{\star}\lr{\tr\tau_h}}{\db{\tr\star d^-\tau_h^c}}\lesssim\|h^{-\frac{1}{2}}\lr{\tr\tau_h}\|_{\mathcal{E}^o_h}\|d^-\tau_h^c\|\\
    &\quad\leq C_{\rho_0}\rho^{-\frac{1}{2}}\|h^{-\frac{1}{2}}\lr{\tr\tau_h}\|_{\mathcal{E}^o_h}\big(\|d_h^-\tau_h\|+\|d_h^-(\tau_h-\tau_h^c)\|\big)\\
    &\quad\leq\frac{1}{4}\|d_h^-\tau_h\|^2+ C_{\rho_0}\rho^{-1}\|h^{-\frac{1}{2}}\lr{\tr\tau_h}\|^2_{\mathcal{E}^o_h}.
\end{align*}
Using the previous estimate, the elementary inequality 
\begin{align*}
    &\|d_h^-\tau_h\|^2\leq2(d_h^-\tau_h,d^-\tau^c_h)+\|d_h^-(\tau_h-\tau^c_h)\|^2,
\end{align*}
and \eqref{approxtau}, we have
\begin{align*}
    \|d_h^-\tau_h\|^2\leq C_{\rho_0}\rho^{-1}\|h^{-\frac{1}{2}}\lr{\tr\tau_h}\|^2_{\mathcal{E}^o_h}.
\end{align*}
Combining it with \eqref{ahlower} eventually proves the coercivity.
\end{proof}

Introducing the discrete bilinear form
\begin{align*}
\tilde{B}_h(\tilde{\sigma}_h,u_h;\tilde{\tau}_h,v_h)=a_h(\tilde{\sigma}_h;\tilde{\tau}_h)+b_h(\tilde{\sigma}_h,v_h)+b_h(\tilde{\tau}_h,u_h)+c_h(u_h,v_h),
\end{align*}
the reduced XG method \eqref{XGreduced} could be recast into 
\begin{equation}
    \tilde{B}_h(\tilde{\sigma}_h,u_h;\tilde{\tau}_h,v_h)=-(f,v_h),\quad\forall  (\tilde{\tau}_h,v_h)\in\Sigma_h\times\mathcal{V}_h.
\end{equation}
Now we are in a position to present the first main result.
\begin{theorem}\label{discreteinfsup}
Let $\rho_0>0$ be a constant, $\rho\in(0,\rho_0]$, and \begin{equation*}
    \textsf{a}\simeq-\rho h,\quad\textsf{b}\simeq\rho^{-1}h^{-1},\quad\textsf{c}\simeq\rho h,\quad\textsf{d}\simeq-\rho^{-1}h^{-1},
\end{equation*} 
in \eqref{check}. It holds that for $\tilde{\sigma}_h, \tilde{\tau}_h\in\Sigma_h$, $u_h, v_h\in\mathcal{V}_h,$
\begin{equation}\label{bdd}
    \tilde{B}_h(\tilde{\sigma}_h,u_h;\tilde{\tau}_h,v_h)\leq C_{\emph{bd}}\vertiii{(\tilde{\sigma}_h,u_h)}_{\rho ,h}\vertiii{(\tilde{\tau}_h,v_h)}_{\rho ,h},
\end{equation}
where $C_{\emph{bd}}>0$ is a constant dependent on $\Omega$, $C_{\emph{mesh}},$ $\{r_i\}_{i=1}^7.$
In addition, let the assumptions \eqref{CheckAssumption} and \eqref{Assumption2} hold and $d_h\mathcal{V}_h\subseteq\mathcal{V}_h^+$. Then there exists a constant $C_{\emph{stab}}$ relying on $\Omega$, $C_{\emph{mesh}}$, $\rho_0$, $\{r_i\}_{i=1}^7,$ such that
\begin{equation*}
    \inf_{\substack{(\tilde{\tau}_h,v_h)\in\Sigma_h\times\mathcal{V}_h\\(\tilde{\tau}_h,v_h)\neq0}}\sup_{\substack{(\tilde{\sigma}_h,u_h)\in\Sigma_h\times\mathcal{V}_h\\(\tilde{\sigma}_h,u_h)\neq0}}\frac{\tilde{B}_h(\tilde{\sigma}_h,u_h;\tilde{\tau}_h,v_h)}{\vertiii{(\tilde{\sigma}_h,u_h)}_{\rho,h}\vertiii{(\tilde{\tau}_h,v_h)}_{\rho,h}}\geq C_\emph{stab}>0,
\end{equation*}
\end{theorem}
\begin{proof}
The boundedness of $\tilde{B}_h$ is a direct consequence of the trace inequality \eqref{trace} and the Cauchy--Schwarz inequality. The inf-sup condition of $\tilde{B}_h$ follows from Lemmas  \ref{lemmabh} and \ref{lemmaah} and $\textsf{a}\leq0,$ see, e.g., Theorem 5.5.1 in \cite{BoffiBrezziFortin2013} for the theory of inf-sup condition of perturbed saddle point systems.
\end{proof}

Theorem \ref{discreteinfsup} also implies the stability of equivalent XG schemes \eqref{XGlocal} and \eqref{XGcompact}. Next we present a priori error estimates of the XG method. 
\begin{corollary}
Let the assumptions in Theorem \ref{discreteinfsup} hold. Then we have
\begin{equation}\label{optimal}
\begin{aligned}
        &\vertiii{\sigma-\sigma_h}_{\rho,h}+\|\xi-\xi_h\|+\vertiii{u-u_h}_{\rho,h}\\
        &\leq C_{\rho_0}\inf_{\tau_h\in\mathcal{V}_h^-,\xi_h\in\mathcal{V}^+_h,v_h\in\mathcal{V}_h}\big(\vertiii{\sigma-\tau_h}_{\rho,h}+\|\xi-\eta_h\|+\vertiii{u-v_h}_{\rho,h}\big).
\end{aligned}
\end{equation}
In addition, assume that
\begin{itemize}
    \item  $\mathcal{P}_{r+1}^-\Lambda^{k-1}(\mathcal{T}_h)\subseteq\mathcal{V}_h^-$, $\mathcal{P}^-_{r+1}\Lambda^{k}(\mathcal{T}_h)\subseteq\mathcal{V}_h$, $\tilde{\mathcal{P}}_r\Lambda^{k+1}(\mathcal{T}_h)\subseteq\mathcal{V}_h^+$,
    \item $u, d^-\sigma\in H^s\Lambda^k(\Omega),$ $\sigma\in H^s\Lambda^{k-1}(\Omega),$ $\xi\in H^s\Lambda^{k+1}(\Omega)$ for an index $s>0$.
\end{itemize}
Then we have for $0< t\leq\min(r+1,s)$,
\begin{equation}\label{rate}
    \begin{aligned}
    &\vertiii{\sigma-\sigma_h}_{\rho,h}+\|\xi-\xi_h\|+\vertiii{u-u_h}_{\rho,h}\\
    &\quad\leq C_{\rho_0} h^t\big(|\sigma|_{H^t\Lambda^{k-1}(\Omega)}+|d^-\sigma|_{H^t\Lambda^k(\Omega)}+|\xi|_{H^t\Lambda^{k+1}(\Omega)}+|u|_{H^t\Lambda^k(\Omega)}\big).
\end{aligned}
\end{equation}
\end{corollary}
\begin{proof}
Let $\tilde{\sigma}=(\sigma,\xi,0)\in V^-\times W^+\times L^2\Lambda^{n-k-1}(\mathcal{E}_h).$
It follows from the inf-sup condition in Theorem \ref{discreteinfsup} and the consistency in Remark \ref{consistency} that 
\begin{equation*}
    \vertiii{(\tilde{\sigma}-\tilde{\sigma}_h,u-u_h)}_{\rho,h}\leq \frac{C_{\text{bd}}}{C_{\text{stab}}}\inf_{(\tilde{\tau}_h,v_h)\in\Sigma_h\times\mathcal{V}_h}\vertiii{(\tilde{\sigma}-\tilde{\tau}_h,u-v_h)}_{\rho,h},
\end{equation*}
see, e.g., \cite{XuZikatanov2003}.
As a result, we obtain \eqref{optimal}.
Let $\pi^-_h: L^2\Lambda^{k-1}(\Omega)\rightarrow V^-_h\subset\mathcal{V}^-_h$, $\pi_h: L^2\Lambda^k(\Omega)\rightarrow V_h\subset\mathcal{V}_h$ be smoothed commuting projections (cf.~\cite{ArnoldFalkWinther2006}) in FEEC, and $Q^+_h$ be the $L^2$-projection onto $\mathcal{V}_h^+$. A direct consequence of \eqref{optimal} is
\begin{equation}\label{rate1}
\begin{aligned}
        &\vertiii{\sigma-\sigma_h}_{\rho,h}+\|\xi-\xi_h\|+\vertiii{u-u_h}_{\rho,h}\\
        &\leq C_{\rho_0}\big(\vertiii{\sigma-\pi^-_h\sigma}_{\rho,h}+\|\xi-Q_h^+\xi\|+\vertiii{u-\pi_hu}_{\rho,h}\big)\\
        &=C_{\rho_0}\big(\|\sigma-\pi^-_h\sigma\|_{V^-}+\|\xi-Q_h^+\xi\|+\|u-\pi_hu\|_V\big).
\end{aligned}
\end{equation}
We complete the proof by \eqref{rate1} and the standard best approximation property of $\pi^-_h$, $\pi_h$, and $Q_h^+.$
\end{proof}
Using \eqref{rate} and the regularity \eqref{regularity}, we have
\begin{align*}
    \vertiii{\sigma-\sigma_h}_{\rho,h}+\|\xi-\xi_h\|+\vertiii{u-u_h}_{\rho,h}=\mathcal{O}(h^{s_0}),\quad s_0\in\left(\frac{1}{2},1\right],
\end{align*}
on arbitrary Lipschitz polyhedron $\Omega.$

\section{Relationship with conforming mixed methods}\label{SectionConnection}

\subsection{An alternative discretization}
The AFW method \eqref{AFW} and the XG method \eqref{XGreduced} with parameters specified in Theorem \ref{discreteinfsup} essentially utilize the conformity 
\begin{equation*}
    u\in H\Lambda^k(\Omega),\quad\sigma=\delta u\in H\Lambda^{k-1}(\Omega).
\end{equation*}
As mentioned in \eqref{conformity2}, it also holds that
\begin{equation}
    u\in\mathring{H}^*\Lambda^{k}(\Omega),\quad\xi=du\in\mathring{H}^*\Lambda^{k+1}(\Omega),
\end{equation}
which suggests an XG method with different parameters in the discrete level. The starting point is the following variational formulation of \eqref{1storder}:
\begin{equation}\label{mixedHodgeLaplace2}
\begin{aligned}
(\xi,\eta)+(\delta^+\eta,u)&=0,\quad&&\eta\in \mathring{H}^*\Lambda^{k+1}(\Omega),\\
(\sigma,\tau)+(\delta u,\tau)&=0,\quad&&\tau\in L^2\Lambda^{k-1}(\Omega),\\
(\delta^+\xi,v)+(\sigma,\delta v)&=-(f,v),&& v\in \mathring{H}^*\Lambda^{k}(\Omega).
\end{aligned}
\end{equation}
In fact \eqref{mixedHodgeLaplace2} is based on the dual complex \eqref{dualdeRham}, which is \emph{isomorphic} to the de Rham complex \eqref{deRhamessential} via the identification \begin{equation*}
\begin{CD}
\mathring{H}\Lambda^{n-k}(\Omega)@>\star>>\mathring{H}^*\Lambda^k(\Omega).
\end{CD}
\end{equation*}
Therefore, \eqref{mixedHodgeLaplace2} is nothing but the variational formulation \eqref{mixedHodgeLaplace} for the Hodge Laplacian with index $n-k$ under essential boundary conditions.

To discretize \eqref{mixedHodgeLaplace2}, we choose 
\begin{equation}
    \mathring{V}_h^{\star+}:=\left\{\begin{aligned}
        &\star\mathring{\mathcal{P}}_{r+1}\Lambda^{n-k-1}(\Th),\\
        &\star\mathring{\mathcal{P}}^-_{r+1}\Lambda^{n-k-1}(\Th)
    \end{aligned}\right\},\quad \mathring{V}^\star_h:=\left\{\begin{aligned}
        &\star\mathring{\mathcal{P}}^-_{r+1}\Lambda^{n-k}(\Th),\\
        &\star\mathring{\mathcal{P}}_{r}\Lambda^{n-k}(\Th)
    \end{aligned}\right\}.
\end{equation}
The conforming mixed method for  \eqref{mixedHodgeLaplace2} is to find $(\mathring{\xi}^c_{h},\mathring{u}^c_{h})\in\mathring{V}_h^{\star+}\times\mathring{V}_h^\star$, such that
\begin{equation}\label{AFW2}
\begin{aligned}
(\mathring{\xi}^c_{h},\eta_h)+(\delta^+\eta_h,\mathring{u}^c_{h})&=0,&&\eta_h\in\mathring{V}_{h}^{\star+},\\
(\delta^+\mathring{\xi}^c_{h},v_h)-(\delta \mathring{u}^c_{h},\delta v_h)&=-(f,v_h),&&v_h\in\mathring{V}_{h}^\star.
\end{aligned}
\end{equation}
The variable $\mathring{\sigma}_h^c$ is directly defined as $\mathring{\sigma}_h^c=-\delta\mathring{u}_h^c.$

We introduce the following notation in a fashion similar to Section \ref{SectionInfsup}. \begin{align*}
&\Sigma^*_h=\CV^-_h\times\check{\CV}_h^-\times\CV_h^+,\\
    &\tilde{\xi}_h=(\sigma_h,\check{\sigma}_h,\xi_h)\in\Sigma^*_h,\quad\tilde{\eta}_h=(\tau_h,\check{\tau}_h,\eta_h)\in\Sigma^*_h,\\
       &a^*_h(\tilde{\xi}_h,\tilde{\eta}_h)=(\sigma_h,\tau_h)+(\xi_h,\eta_h)-\langle{\textsf{d}\lr{\tr\sigma_h}},\lr{\tr\tau_h}\rangle\\
       &\qquad+\langle\textsf{c}\lr{\tr\star\xi_h},\lr{\tr\star\eta_h}\rangle-\ab{\textsf{a}^{-1}\check{\sigma}_h}{\check{\tau}_h},\\
       &b^*_h(\tilde{\xi}_h,v_h)=(\sigma_h,\delta_hv_h)+(\delta_h^+\xi_h,v_h)+\langle\bar{\star}\db{\tr\sigma_h},\lr{\tr\star v_h}\rangle\\
       &\qquad+\langle\bar{\star}\check{\sigma}_h,\lr{\tr\star v_h}\rangle+\langle{\bar{\star}\db{\tr v_h}},\lr{\tr\star\xi_h}\rangle,\\
       &c^*_h(u_h,v_h)=-\langle\textsf{b}\lr{\tr u_h},\lr{\tr v_h}\rangle.
    \end{align*}
Furthermore, given $\eta_h\in\mathcal{V}^+_h$, $v_h\in\mathcal{V}_h$, $\tilde{\eta}_h\in\Sigma^*_h$, and a positive constant $\rho$, we propose the following mesh dependent norms
\begin{align*}
           \vertiii{\eta_h}_{*,\rho,h}^2&:=\|\eta_h\|^2+\|\delta^+_h\eta_h\|^2+\rho^{-1}\|h^{-\frac{1}{2}}\lr{\tr \star\eta_h}\|_{\mathcal{E}_h}^2,\\ \vertiii{v_h}_{*,\rho,h}^2&:=\|v_h\|^2+\|\delta_hv_h\|^2+\rho^{-1}\|h^{-\frac{1}{2}}\lr{\tr\star v_h}\|_{\mathcal{E}_h}^2,\\
    \vertiii{\tilde{\eta}_h}^2_{*,\rho ,h}&:=\vertiii{\eta_h}^2_{*,\rho,h}+\|\tau_h\|^2+\rho\|h^\frac{1}{2}\check{\tau}_h\|_{\mathcal{E}_h}^2,\\
    \vertiii{(\tilde{\eta}_h,v_h)}^2_{*,\rho ,h}&:=\vertiii{\tilde{\eta}_h}^2_{*,\rho ,h}+\|v_h\|_{*,\rho,h}^2.
\end{align*} 
The four-field formulation \eqref{XGreducedII} is to find $(\tilde{\xi}_h,u_h)\in\Sigma_h^*\times\mathcal{V}_h$ such that
\begin{equation}\label{XGreduced2}
    \begin{aligned}
    a^*_h(\tilde{\xi}_h,\tilde{\eta}_h)+b^*_h(\tilde{\eta}_h,u_h)&=0,\quad\tilde{\eta}_h\in\Sigma_h^*,\\
    b^*_h(\tilde{\xi}_h,v_h)+c^*_h(u_h,v_h)&=-(f,v_h),\quad v_h\in\mathcal{V}_h.
    \end{aligned}
\end{equation}
The total bilinear form is defined as
\begin{equation*}
    \tilde{B}^*_h(\tilde{\xi}_h,u_h;\tilde{\eta}_h,v_h)=a^*_h(\tilde{\xi}_h,\tilde{\eta}_h)+b^*_h(\tilde{\xi}_h,v_h)+b^*_h(\tilde{\eta}_h,u_h)+c^*_h(u_h,v_h).
\end{equation*}

The next theorem confirms the well-posedness of \eqref{XGreduced}. Its proof follows from the same line in Lemmas \ref{lemmabh} and \ref{lemmaah} and Theorem \ref{discreteinfsup} except that $a_h$, $b_h$, $c_h$  are replaced with $a^*_h$, $b^*_h$, $c^*_h$. The analysis tools are the dual complex \eqref{dualdeRham} and the discrete dual complex by $(\star\mathring{V}_h,\delta)$.
\begin{theorem}\label{discreteinfsupII}
Let $\rho_0>0$ be a constant, $\rho\in(0,\rho_0]$, and \begin{equation*}
    \textsf{a}\simeq-\rho ^{-1}h^{-1},\quad\textsf{b}\simeq\rho h,\quad\textsf{c}\simeq\rho^{-1} h^{-1},\quad\textsf{d}\simeq-\rho h
\end{equation*}
in \eqref{check}. It holds that for $\tilde{\xi}_h, \tilde{\eta}_h\in\Sigma^*_h$, $u_h, v_h\in\mathcal{V}_h,$
\begin{equation*}
    \tilde{B}^*_h(\tilde{\xi}_h,u_h;\tilde{\eta}_h,v_h)\leq C_{\rho_0}\vertiii{(\tilde{\xi}_h,u_h)}_{*,\rho ,h}\vertiii{(\tilde{\eta}_h,v_h)}_{*,\rho ,h}.
\end{equation*}
In addition, assume $\delta_h\mathcal{V}_h\subseteq\mathcal{V}_h^-$ and
\begin{equation*}
\mathcal{V}_h^{k+1}\cap \mathring{H}^*\Lambda^{k+1}(\Omega)=\star\mathring{V}_h^{n-k-1},\quad\mathcal{V}_h^{k}\cap \mathring{H}^*\Lambda^{k}(\Omega)=\star\mathring{V}_h^{n-k}.
\end{equation*}
Then we have
\begin{equation*}
    \inf_{\substack{(\tilde{\xi}_h,u_h)\in\Sigma^*_h\times\mathcal{V}_h\\(\tilde{\xi}_h,u_h)\neq0}}\sup_{\substack{(\tilde{\eta}_h,v_h)\in\Sigma^*_h\times\mathcal{V}_h\\(\tilde{\eta}_h,v_h)\neq0}}\frac{\tilde{B}^*_h(\tilde{\xi}_h,u_h;\tilde{\eta}_h,v_h)}{\vertiii{(\tilde{\xi}_h,u_h)}_{*,\rho,h}\vertiii{(\tilde{\eta}_h,v_h)}_{*,\rho,h}}= C_{\rho_0}>0.
\end{equation*}
\end{theorem}

\subsection{Limiting Cases}
The XG method is a departure from the AFW conforming mixed methods \eqref{AFW} and \eqref{AFW2}. In this subsection, we show that the XG solution $(\sigma_h,\xi_h,u_h)$ converges to the conforming solution $(\sigma^c_h,\xi^c_h,u^c_h)$ or $(\mathring{\sigma}^c_h,\mathring{\xi}^c_h,\mathring{u}^c_h)$ when passing penalty parameters to suitable limits. A key ingredient in the analysis is the stability of the AFW mixed method \eqref{AFW}
\begin{equation}\label{AFWstability}
\|\sigma^c_{h}\|_{V^-}+\|\xi^c_{h}\|+\|u_h^c\|_V\lesssim\|f\|.
\end{equation}
\begin{theorem}\label{theoremlimit}
Let the assumptions in Theorem \ref{discreteinfsup} hold. Then we have 
\begin{equation}\label{limitI}
    \|\sigma_h-\sigma_h^c\|+\|d_h^-(\sigma_h-\sigma_h^c)\|+\|u_h-u_h^c\|+\|d_h(u_h-u_h^c)\|\lesssim\rho^{\frac{1}{2}}\|f\|.
\end{equation}
Let the assumptions in Theorem \ref{discreteinfsupII} hold. Then we have 
\begin{equation}\label{limitII}
    \|\xi_h-\mathring{\xi}_h^c\|+\|\delta_h^+(\xi_h-\mathring{\xi}_h^c)\|+\|u_h-\mathring{u}_h^c\|+\|\delta_h(u_h-\mathring{u}_h^c)\|\lesssim\rho^{\frac{1}{2}}\|f\|.
\end{equation}
\end{theorem}
\begin{proof}
We focus on \eqref{limitI} and the proof of \eqref{limitII} is similar. Let $\delta_\sigma=\sigma_h-\sigma_h^c\in\mathcal{V}_h^-$, $\delta_u=u_h-u_h^c\in\mathcal{V}_h$, $\delta_\xi=\xi_h+du_h^c\in\mathcal{V}_h^+$.
Rewriting the three-field formulation \eqref{XGcompact} leads to
\begin{equation*}
    \begin{aligned}
       &(\delta_\sigma,\tau_h)+(d^-_h\tau_h,\delta_u)-\langle{\bar{\star}\lr{\tr\tau_h}},\db{\tr\delta_u}\rangle-\langle{\textsf{d}\lr{\tr\delta_\sigma}},\lr{\tr\tau_h}\rangle\\
       &\qquad=-(\sigma_h^c,\tau_h)-(d^-_h\tau_h,u_h^c)-\langle{\bar{\star}\lr{\tr\tau_h}},\db{\tr\star{u}^c_h}\rangle,\\
       &(\delta_\xi,\eta_h)+(\eta_h,\delta_u)+\langle\textsf{c}\lr{\tr \star\delta_\xi},\lr{\tr\star\eta_h}\rangle-\langle\bar{\star}\lr{\tr \delta_u},\db{\tr\star\eta_h}\rangle\\
       &\qquad=-(\xi_h^c,\eta_h)-(\eta_h,du_h^c)-\langle\textsf{c}\lr{\tr \star\xi^c_h},\lr{\tr\star\eta_h}\rangle.\\
       &(d_h^-\delta_\sigma,v_h)+(\delta_\xi,d_hv_h)+\langle\textsf{a}\lr{\tr\star \delta_u},\lr{\tr\star v_h}\rangle-\langle\bar{\star}\lr{\tr\delta_\sigma},\db{\tr\star v_h}\rangle\\
       &\qquad-\langle{\bar{\star}\lr{\tr v_h}},\db{\tr\star\delta_\xi}\rangle-\langle{\textsf{b}\lr{\tr \delta_u}},\lr{\tr v_h}\rangle=-(f,v_h)-(d_h^-\sigma_h^c,v_h)\\
       &\qquad+(du_h^c,d_hv_h)-\langle\textsf{a}\lr{\tr\star u_h^c},\lr{\tr\star v_h}\rangle+\langle{\bar{\star}\lr{\tr v_h}},\db{\tr\star\xi_h^c}\rangle,
    \end{aligned}
\end{equation*}
for all $(\tau_h,\eta_h,v_h)\in\mathcal{V}_h^-\times\mathcal{V}_h^+\times\mathcal{V}_h$. We take $v_h^c\in V_h$ $\tau_h^c\in V_h^-$ in Lemma \ref{approximation}. It then follows from the previous equations and \eqref{AFW} that
\begin{equation}\label{defect}
    \begin{aligned}
       &(\delta_\sigma,\tau_h)+(d^-_h\tau_h,\delta_u)-\langle{\bar{\star}\lr{\tr\tau_h}},\db{\tr\delta_u}\rangle-\langle{\textsf{d}\lr{\tr\delta_\sigma}},\lr{\tr\tau_h}\rangle\\
       &\qquad=-(\sigma_h^c,\tau_h-\tau_h^c)-(d^-_h(\tau_h-\tau_h^c),u_h^c)-\langle{\bar{\star}\lr{\tr\tau_h}},\db{\tr\star{u}^c_h}\rangle,\\
       &(\delta_\xi,\eta_h)+(\eta_h,\delta_u)+\langle\textsf{c}\lr{\tr \star\delta_\xi},\lr{\tr\star\eta_h}\rangle-\langle\bar{\star}\lr{\tr \delta_u},\db{\tr\star\eta_h}\rangle\\
       &\qquad=-\langle\textsf{c}\lr{\tr \star\xi^c_h},\lr{\tr\star\eta_h}\rangle,\\
       &(d_h^-\delta_\sigma,v_h)+(\delta_\xi,d_hv_h)+\langle\textsf{a}\lr{\tr\star \delta_u},\lr{\tr\star v_h}\rangle-\langle\bar{\star}\lr{\tr\delta_\sigma},\db{\tr\star v_h}\rangle\\
       &\quad-\langle{\bar{\star}\lr{\tr v_h}},\db{\tr\star\delta_\xi}\rangle-\langle{\textsf{b}\lr{\tr \delta_u}},\lr{\tr v_h}\rangle\\
       &\qquad=-(f,v_h-v_h^c)-(d_h^-\sigma_h^c,v_h-v_h^c)+(du_h^c,d_h(v_h-v_h^c))\\
       &\qquad-\langle\textsf{a}\lr{\tr\star u_h^c},\lr{\tr\star v_h}\rangle+\langle{\bar{\star}\lr{\tr v_h}},\db{\tr\star\xi_h^c}\rangle.
    \end{aligned}
\end{equation}
Combining \eqref{defect} with the inf-sup condition of \eqref{XGcompact} (implied by Theorem \ref{discreteinfsup}) and using the Cauchy--Schwarz inequality and \eqref{AFWstability}, \eqref{trace}, we have
\begin{equation*}
    \begin{aligned}
        &\vertiii{\delta_\sigma}_{\rho,h}+\|\delta_\xi\|+\vertiii{\delta_u}_{\rho,h}\\
        &\lesssim\sup_{\substack{(\tau_h,\eta_h,v_h)\in\mathcal{V}_h^-\times\mathcal{V}_h^+\times\mathcal{V}_h\\\vertiii{\tau_h}_{\rho,h}=1,\|\eta_h\|=1,\vertiii{v_h}_{\rho,h}=1}}\left\{\|f\|\big(\|\tau_h-\tau_h^c\|+\|d_h^-(\tau_h-\tau_h^c)\|+\|h^{-\frac{1}{2}}\lr{\tr\tau_h}\|_{\mathcal{E}_h^o}\big)\right.\\
        &\qquad\left.+\rho^\frac{1}{2}\|f\|+\|f\|\big(\|v_h-v_h^c\|+\|d_h(v_h-v_h^c)\|+\|h^{-\frac{1}{2}}\lr{\tr v_h}\|_{\mathcal{E}_h^o}\big)\right\}
    \end{aligned}
\end{equation*}
Then it follows from Lemma \ref{approximation} and a Young's inequality that
\begin{equation*}
    \vertiii{\delta_\sigma}_{\rho,h}+\|\delta_\xi\|+\vertiii{\delta_u}_{\rho,h}\lesssim\rho^\frac{1}{2}\|f\|,
\end{equation*}
which completes the proof.
\end{proof}

As a consequence of the above theorem, with parameters  specified in Theorem \ref{discreteinfsup}, the XG solution $(\sigma_h,\xi_h,u_h)$ converges to $(\sigma^c_h,\xi^c_h,u^c_h)$ as  $\rho\rightarrow0$.
On the other hand, taking parameters in Theorem \ref{discreteinfsupII}, the XG solution $(\sigma_h,\xi_h,u_h)$ is approaching  $(\mathring{\sigma}^c_h,\mathring{\xi}^c_h,\mathring{u}^c_h)$ as $\rho\rightarrow0.$

\section{Hybridization}\label{SectionHybridization}
In this section, we show that the XG method \eqref{XGlocal} with  appropriate check spaces and  parameters $\textsf{a},$ $\textsf{b},$ $\textsf{c},$ $\textsf{d}$ given in Theorem \ref{discreteinfsup} is hybridizable. In this case, solving \eqref{XGlocal} is equivalent to solving a reduced system with a smaller number of degrees of freedom. In particular, the penalty parameters are taken as
\begin{equation}\label{abcd}
    \textsf{c}=(4\textsf{b})^{-1},\quad\textsf{d}=(4\textsf{a})^{-1}.
\end{equation}

\subsection{Local Solvers}
The first step is to express $\sigma_h$, $\xi_h,$ $u_h$ on each local element $T$ in terms of the fluxes $\hat{u}_h^\star$, $\hat{\xi}_h^\star$. Recall the relations in \eqref{check}
\begin{equation}\label{id1}
    \begin{aligned}
    &\bar{\star}\check{u}_h=\textsf{c}\lr{\tr \star\xi_h},\\
    &\bar{\star}\lr{\tr u_h}=\textsf{b}^{-1}\check{\xi}_h^\star=\textsf{b}^{-1}(\hat{\xi}_h^\star-\bar{\xi}_h^\star).
\end{aligned}
\end{equation}
Then using the elementary identities
\begin{align*}
    &\tr u_h-\bar{u}_h=2^{-1}s_T\lr{\tr u_h}\text{ on }\partial T,\\
    &2^{-1}s_T\lr{\tr\star \xi_h}+\bar{\xi}^\star_h=\tr\star \xi_h\text{ on }\partial T,
\end{align*}
and $\hat{u}_h=\check{u}_h+\bar{u}_h,$ \eqref{id1}, \eqref{abcd}, we obtain 
\begin{equation}\label{uhattraceu}
    \begin{aligned}
    &\langle\bar{\star}(\hat{u}_h-\tr u_h),\tr\star\eta_h\rangle^s_{\partial T}\\
    &=\langle s_T(\textsf{c}\lr{\tr \star\xi_h}-2^{-1}s_T\bar{\star}\lr{\tr u_h}),\tr\star\eta_h\rangle_{\partial T}\\
    &=\langle s_T\textsf{c}\lr{\tr \star\xi_h}-(2\textsf{b})^{-1}(\hat{\xi}_h^\star-\bar{\xi}_h^\star),\tr\star\eta_h\rangle_{\partial T}\\
    &=\langle2\textsf{c}(\tr\star{\xi}_h-\hat{\xi}_h^\star),\tr\star\eta_h\rangle_{\partial T}.
    \end{aligned}
\end{equation}
Similarly, it follows from
\begin{align*}
    &\bar{\star}\check{\sigma}_h=\textsf{a}\lr{\tr\star u_h},\\
    &\bar{\star}\lr{\tr\sigma_h}=\textsf{d}^{-1}\check{u}_h^\star=\textsf{d}^{-1}(\hat{u}_h^\star-\bar{u}_h^\star),\\
    &\tr\sigma_h-\bar{\sigma}_h=2^{-1}s_T\lr{\tr \sigma_h}\text{ on }\partial T,\\
    &2^{-1}s_T\lr{\tr\star u_h}+\bar{u}^\star_h=\tr\star u_h\text{ on }\partial T,
\end{align*}
and $\hat{\sigma}_h=\check{\sigma}_h+\bar{\sigma}_h$, \eqref{abcd} that 
\begin{equation}\label{sigmahattracesigma}
    \begin{aligned}
    &\langle\bar{\star}(\hat{\sigma}_h-\tr{\sigma}_h),{\tr\star v_h}\rangle^s_{\partial T}\\
    &=\langle s_T(\textsf{a}\lr{\tr\star u_h}-2^{-1}s_T\bar{\star}\lr{\tr \sigma_h}),\tr\star v_h\rangle_{\partial T}\\
    &=\langle s_T\textsf{a}\lr{\tr\star u_h}-(2\textsf{d})^{-1}(\hat{u}_h^\star-\bar{u}_h^\star),\tr\star v_h\rangle_{\partial T}\\
    &=\langle2\textsf{a}(\tr\star u_h-\hat{u}_h^\star),\tr\star v_h\rangle_{\partial T}.
    \end{aligned}
\end{equation}

Using \eqref{uhattraceu}, \eqref{sigmahattracesigma},  the local problem \eqref{XGlocal2} becomes
\begin{subequations}\label{localsolver0}
    \begin{align}
       &(\sigma_h,\tau_h)_T+(d^-\tau_h,u_h)_T=\langle{\bar{\star}\tr\tau_h},\hat{u}_h^\star\rangle^s_{\partial T},\label{localsolver1}\\
       &(\xi_h,\eta_h)_T+(\eta_h,du_h)_T+\langle2\textsf{c}\tr\star{\xi}_h,\tr\star\eta_h\rangle_{\partial T}=\langle2\textsf{c}\hat{\xi}_h^\star,\tr\star\eta_h\rangle_{\partial T},\label{localsolver2}\\
       &(d^-\sigma_h,v_h)_T+(\xi_h,dv_h)_T+\langle2\textsf{a}\tr\star u_h,\tr\star v_h\rangle_{\partial T}\label{localsolver3}\\
       &\quad=-(f,v_h)_T+\langle{\bar{\star}\tr v_h},\hat{\xi}_h^\star\rangle^s_{\partial T}+\langle2\textsf{a}\hat{u}_h^\star,\tr\star v_h\rangle_{\partial T}.\nonumber
    \end{align}
\end{subequations}
Define the local bilinear forms
\begin{align*}
  a_T(\sigma_T,\xi_T;\tau_T,\eta_T)&=(\sigma_T,\tau_T)_T+(\xi_T,\eta_T)_T+\langle2\textsf{c}\tr\star{\xi}_T,\tr\star\eta_T\rangle_{\partial T},\\
  b_T(\sigma_T,\xi_T;v_T)&=(d^-\sigma_T,v_T)_T+(\xi_T,dv_T)_T,\\
  c_T(u_T,v_T)&=\langle2\textsf{a}\tr\star u_T,\tr\star v_T\rangle_{\partial T}.
\end{align*}
It follows from \eqref{localsolver0} that $(\sigma_h|_T,\xi_h|_T,u_h|_T)\in\mathcal{V}^-_h|_T\times\mathcal{V}^+_h|_T\times\mathcal{V}_h|_T$ satisfies 
\begin{equation}\label{localsolver}
    \begin{aligned}
    &a_T(\sigma_h,\xi_h;\tau_T,\eta_T)+b_T(\tau_T,\eta_T;u_h)=\langle{\bar{\star}\tr\tau_T},\hat{u}_h^\star\rangle^s_{\partial T}+\langle2\textsf{c}\hat{\xi}_h^\star,\tr\star\eta_T\rangle_{\partial T},\\
    &b_T(\sigma_h,\xi_h;v_T)+c_T(u_h,v_T)=-(f,v_T)_T+\langle{\bar{\star}\tr v_T},\hat{\xi}_h^\star\rangle^s_{\partial T}+\langle2\textsf{a}\hat{u}_h^\star,\tr\star v_T\rangle_{\partial T},
\end{aligned}
\end{equation}
for all $(\tau_T,\eta_T,v_T)\in\mathcal{V}^-_h|_T\times\mathcal{V}^+_h|_T\times\mathcal{V}_h|_T$. 

It is noted that \eqref{localsolver} is
a \emph{continuous} saddle point system on $T$ and $\textsf{a}<0$, $\textsf{c}>0$. Therefore the well-posedness of \eqref{localsolver} follows from the same argument used in the proof of Theorem \ref{continuousinfsup}. Due to the local problem \eqref{localsolver}, on each  $T\in\mathcal{T}_h$, local XG solutions $\sigma_h|_T$, $\xi_h|_T$, $u_h|_T$ are determined by $\hat{u}_h^\star|_{\partial T}$, $\hat{\xi}_h^\star|_{\partial T}$  and we use the notation
\begin{align*}
    \sigma_h=H_h^-(\hat{u}_h^\star,\hat{\xi}_h^\star),\quad\xi_h =H_h^+(\hat{u}_h^\star,\hat{\xi}_h^\star),\quad u_h=H_h(\hat{u}_h^\star,\hat{\xi}_h^\star)
\end{align*}
to indicate such dependence. Therefore we say $H_h^-: \check{\mathcal{V}}_h^\star\times\check{\mathcal{V}}^{+\star}_h\rightarrow\mathcal{V}_h^-$, $H_h: \check{\mathcal{V}}_h^\star\times\check{\mathcal{V}}^{+\star}_h\rightarrow\mathcal{V}_h$, $H_h^+: \check{\mathcal{V}}_h^\star\times\check{\mathcal{V}}^{+\star}_h\rightarrow\mathcal{V}_h^+$ are \emph{local solvers}. 

\subsection{Global Coupled System}
It remains to derive a reduced global system for the two variables $\hat{u}_h^\star$, $\hat{\xi}_h^\star$. Let the check spaces satisfy
\begin{equation}\label{checkhybrid}
       \tr\star\mathcal{V}_h\subseteq\check{\mathcal{V}}_h^\star,\quad\tr\star\mathcal{V}^+_h\subseteq\check{\mathcal{V}}^{+\star}_h,
\end{equation}
which implies $\hat{u}_h^\star\in\check{\mathcal{V}}_h^\star$, $\hat{\xi}_h^\star\in\check{\mathcal{V}}_h^{+\star}$.

Using \eqref{meanjump} and \eqref{abcd}, we rewrite \eqref{checkd}, \eqref{checkb} as
\begin{subequations}
\begin{align}
 &\ab{2\textsf{a}(\hat{u}_h^\star-\tr\star u_h)}{\hat{v}_h^\star}_{\partial \mathcal{T}_h}-\ab{\bar{\star}\tr\sigma_h}{\hat{v}_h^\star}^s_{\partial \mathcal{T}_h}=0,\quad\forall\hat{v}_h^\star\in\check{\mathcal{V}}_h^\star,\label{globalflux1}\\
    &\ab{2\textsf{c}(\hat{\xi}_h^\star-\tr\star\xi_h)}{\hat{\eta}_h^\star}_{\partial \mathcal{T}_h}-\ab{\bar{\star}\tr u_h}{\hat{\eta}_h^\star}^s_{\partial \mathcal{T}_h}=0,\quad\forall\hat{\eta}_h^\star\in\check{\mathcal{V}}_h^{+\star},\label{globalflux2}
\end{align}
\end{subequations}
which determine the fluxes $\hat{u}_h^\star$, $\hat{\xi}_h^\star$.

Given $\hat{v}_h^\star\in\check{\mathcal{V}}_h^\star$, $\hat{\eta}_h^\star\in\check{\mathcal{V}}^{+\star}_h$, for the time being, let
\begin{align*}
    \tau_h=H_h^-(\hat{v}_h^\star,\hat{\eta}_h^\star),\quad\eta_h =H_h^+(\hat{v}_h^\star,\hat{\eta}_h^\star),\quad v_h=H_h(\hat{v}_h^\star,\hat{\eta}_h^\star).
\end{align*}
In view of \eqref{localsolver0} or \eqref{localsolver}, $(\tau_h,\eta_h,v_h)$ satisfies 
\begin{subequations}
    \begin{align}
       &(\tau_h,\sigma_h)_T+(d^-\sigma_h,v_h)_T=\langle{\bar{\star}\tr\sigma_h},\hat{v}_h^\star\rangle^s_{\partial T},\label{ls1}\\
       &(\eta_h,\xi_h)_T+(\xi_h,dv_h)_T+\langle2\textsf{c}\tr\star{\eta}_h,\tr\star\xi_h\rangle_{\partial T}=\langle2\textsf{c}\hat{\eta}_h^\star,\tr\star\xi_h\rangle_{\partial T},\label{ls2}\\
       &(d^-\tau_h,u_h)_T+(\eta_h,du_h)_T+\langle2\textsf{a}\tr\star v_h,\tr\star u_h\rangle_{\partial T}\label{ls3}\\
       &\quad=-(f,u_h)_T+\langle{\bar{\star}\tr u_h},\hat{\eta}_h^\star\rangle^s_{\partial T}+\langle2\textsf{a}\hat{v}_h^\star,\tr\star u_h\rangle_{\partial T}.\nonumber
    \end{align}
\end{subequations}
 Using \eqref{globalflux1}, \eqref{ls1}, \eqref{localsolver3}, we have
\begin{equation*}
    \begin{aligned}
       &\ab{2\textsf{a}(\hat{u}_h^\star-\tr\star u_h)}{\hat{v}_h^\star}_{\partial \mathcal{T}_h}\\
       &=\ab{\bar{\star}\tr\sigma_h}{\hat{v}_h^\star}^s_{\partial \mathcal{T}_h}=(\tau_h,\sigma_h)+(d_h^-\sigma_h,v_h)\\
       &=(\tau_h,\sigma_h)-(\xi_h,dv_h)-\langle2\textsf{a}\tr\star u_h,\tr\star v_h\rangle_{\partial\mathcal{T}_h}\\
       &-(f,v_h)+\langle{\bar{\star}\tr v_h},\hat{\xi}_h^\star\rangle^s_{\partial \mathcal{T}_h}+\langle2\textsf{a}\hat{u}_h^\star,\tr\star v_h\rangle_{\partial\mathcal{T}_h}.
    \end{aligned}
\end{equation*}
Simplifying the above equation yields
\begin{equation}\label{global1}
    \begin{aligned}
       &\ab{2\textsf{a}(\hat{u}_h^\star-\tr\star u_h)}{\hat{v}_h^\star-\tr\star v_h}_{\partial \mathcal{T}_h}\\
       &=(\tau_h,\sigma_h)-(\xi_h,dv_h)-(f,v_h)+\langle{\bar{\star}\tr v_h},\hat{\xi}_h^\star\rangle^s_{\partial T}.
    \end{aligned}
\end{equation}
On the other hand, it follows from \eqref{globalflux2} and \eqref{ls2} that
\begin{equation}\label{global2}
    \begin{aligned}
      &\ab{2\textsf{c}\hat{\xi}_h^\star}{\hat{\eta}_h^\star}_{\partial \mathcal{T}_h}-\ab{\bar{\star}\tr u_h}{\hat{\eta}_h^\star}^s_{\partial \mathcal{T}_h}=\ab{2\textsf{c}\tr\star\xi_h}{\hat{\eta}_h^\star}_{\partial \mathcal{T}_h}\\
      &=(\eta_h,\xi_h)+(\xi_h,d_hv_h)+\langle2\textsf{c}\tr\star{\eta}_h,\tr\star\xi_h\rangle_{\partial \mathcal{T}_h}.
    \end{aligned}
\end{equation}
Summing \eqref{global1} and \eqref{global2} leads to
\begin{equation}\label{HDG0}
    \begin{aligned}
       &\ab{2\textsf{a}\big(\hat{u}_h^\star-\tr\star u_h\big)}{\hat{v}_h^\star-\tr\star v_h}_{\partial \mathcal{T}_h}+\ab{2\textsf{c}\hat{\xi}_h^\star}{\hat{\eta}_h^\star}_{\partial \mathcal{T}_h}-\ab{\bar{\star}\tr u_h}{\hat{\eta}_h^\star}^s_{\partial \mathcal{T}_h}\\
       &=(\tau_h,\sigma_h)-(f,v_h)+\langle{\bar{\star}\tr v_h},\hat{\xi}_h^\star\rangle^s_{\partial T}+(\eta_h,\xi_h)+\langle2\textsf{c}\tr\star{\eta}_h,\tr\star\xi_h\rangle_{\partial\mathcal{T}_h}.
    \end{aligned}
\end{equation}
Therefore, \eqref{HDG0} translates into the hybridized globally coupled problem: Find $\hat{u}_h^\star\in\check{\mathcal{V}}^\star_h$, $\hat{\xi}_h^\star\in\check{\mathcal{V}}^{+\star}_h$ such that
\begin{equation}\label{HDG}
    \begin{aligned}
       &(H_h^-(\hat{u}_h^\star,\hat{\xi}_h^\star),H_h^-(\hat{v}_h^\star,\hat{\eta}_h^\star))+(H_h^+(\hat{u}_h^\star,\hat{\xi}_h^\star),H_h^+(\hat{v}_h^\star,\hat{\eta}_h^\star))\\
       &+\langle2\textsf{c}\tr\star H_h^+(\hat{u}_h^\star,\hat{\xi}_h^\star),\tr\star H_h^+(\hat{v}_h^\star,\hat{\eta}_h^\star)\rangle_{\partial\mathcal{T}_h}-\ab{2\textsf{c}\hat{\xi}_h^\star}{\hat{\eta}_h^\star}_{\partial \mathcal{T}_h}\\
       &-\ab{2\textsf{a}(\hat{u}_h^\star-\tr\star H_h(\hat{u}_h^\star,\hat{\xi}_h^\star))}{\hat{v}_h^\star-\tr\star H_h(\hat{v}_h^\star,\hat{\eta}_h^\star)}_{\partial \mathcal{T}_h}\\
       &+\langle{\bar{\star}\tr H_h(\hat{v}_h^\star,\hat{\eta}_h^\star)},\hat{\xi}_h^\star\rangle^s_{\partial \mathcal{T}_h}+\ab{\bar{\star}\tr H_h(\hat{u}_h^\star,\hat{\xi}_h^\star)}{\hat{\eta}_h^\star}^s_{\partial \mathcal{T}_h}=(f,v_h),
    \end{aligned}
\end{equation}
for all $\hat{v}_h^\star\in\check{\mathcal{V}}^\star_h$, $\hat{\eta}_h^\star\in\check{\mathcal{V}}^{+\star}_h$.

\section{Examples}\label{SectionExamples}
This section is devoted to special cases of the abstract XG method \eqref{XGlocal} in $\mathbb{R}^3$. In doing so, we introduce the well-known spaces
\begin{align*}
  &H^1(\Omega)=\{v\in L^2(\Omega): \nabla v\in [L^2(\Omega)]^3\},\\
  &H(\curl,\Omega)=\{v\in L^2(\Omega): \nabla\times v\in [L^2(\Omega)]^3\},\\
&H(\text{div},\Omega)=\{v\in L^2(\Omega): \nabla\cdot v\in L^2(\Omega)\}.
\end{align*}
Using proxy vector fields, the de Rham complex \eqref{deRham} with $n=3$ is identified as the classical exact sequence in $\mathbb{R}^3$ via the following commuting diagram
\begin{equation}\label{identification}
\begin{CD}
    H\Lambda^0(\Omega)@>{d^0}>>H\Lambda^1(\Omega)@>{d^1}>>H\Lambda^2(\Omega)@>{d^2}>>H\Lambda^3(\Omega)\\
    @VV\cong V  @VV\cong V  @VV\cong V @VV\cong V\\
    H^1(\Omega)@>{\nabla}>>H(\curl,\Omega)@>{\nabla\times}>>H(\text{div},\Omega)@>{\nabla\cdot}>>L^2(\Omega)
    \end{CD}
\end{equation}

The 3d Hodge star operator $\star$ is simply identity mapping. On a 2-dimensional face $E\in\mathcal{E}_h$, the 2d Hodge stars $\bar{\star}: L^2\Lambda^0(E)\rightarrow L^2\Lambda^2(E)$ and $\bar{\star}: L^2\Lambda^2(E)\rightarrow L^2\Lambda^0(E)$ are identity operators while $\bar{\star}: L^2\Lambda^1(E)\rightarrow L^2\Lambda^1(E)$ is the mapping $\bar{\star}v=\nu_E\times v.$ For each $E\in\mathcal{E}_h,$ the trace operator $\tr$ is realized as 
\begin{align*}
    &\tr_Ev_0=v_0|_E,\quad v_0\in H^1(\Omega),\\
    &\tr_Ev_1=v_1\times\nu_E,\quad v_1\in H(\text{curl},\Omega),\\
    &\tr_Ev_2=v_2\cdot\nu_E,\quad v_2\in H(\text{div},\Omega).
\end{align*}

In the discrete level, subspaces $\{V_h^k\}_{k=0}^3$ form the discrete de Rham complex \eqref{discretedeRham}, where $V_h^0$ is the Lagrange element space, $V_h^1$ is the N\'ed\'elec edge element space \cite{Nedelec1980,Nedelec1986}, $V_h^2$ is the face element space \cite{RaviartThomas1977,Nedelec1986}, and $V_h^3$ is the space of discontinuous piecewise polynomials.
We refer readers to \cite{ArnoldFalkWinther2006} for more details.
Let $\mathcal{P}_r(\mathcal{T}_h)=\tilde{\mathcal{P}}_r\Lambda^0(\mathcal{T}_h)$ (resp.~$\mathcal{P}_r(\mathcal{E}_h)=\mathcal{P}_r\Lambda^0(\mathcal{E}_h)$) be the space of discontinuous and piecewise polynomials of degree at most $r$ on $\mathcal{T}_h$ (resp.~$\mathcal{E}_h$). Let $\bm{x}=(x_1,x_2,x_3)^\top$ denote the position vector in $\mathbb{R}^3$, and
\begin{align*}
    &\mathring{\mathcal{P}}_r(\mathcal{E}_h)=\{v_h\in\mathcal{P}_r(\mathcal{E}_h): v_h|_E=0,~\forall E\in\mathcal{E}^\partial_h\},\\
    &\mathcal{P}^{\parallel}_r(\mathcal{E}_h)=\{v_h\in[\mathcal{P}_r(\mathcal{E}_h)]^3: v_h|_E\text{ is parallel to }E,~\forall E\in\mathcal{E}_h\},\\
    &\mathring{\mathcal{P}}^{\parallel}_r(\mathcal{E}_h)=\{v_h\in\mathcal{P}^{\parallel}_r(\mathcal{E}_h): v_h|_E=0,~\forall E\in\mathcal{E}^\partial_h\}.
\end{align*}
In Tables \ref{tableDGI}--\ref{tablecheckII}, we list several possible DG spaces used in \eqref{XGlocal}. It is straightforward to check that spaces in Tables \ref{tableDGI} and \ref{tablecheckI} and spaces in Tables \ref{tableDGII} and \ref{tablecheckII} satisfy \eqref{checkhybrid} and  assumptions in Theorem \ref{discreteinfsup}.
\begin{table}[tbhp]
\caption{DG spaces using complete polynomials}
\label{tableDGI}
\centering
\begin{tabular}{|c|c|c|c|}
\hline
{$k$} &  {$\mathcal{V}_h^-$}
 & {$\mathcal{V}_h$}  
& {$\mathcal{V}_h^+$} \\

\hline
               0    & N/A &   $\mathcal{P}_{r+1}(\mathcal{T}_h)$ & $[\mathcal{P}_r(\mathcal{T}_h)]^3$\\
             1       &   $\mathcal{P}_{r+1}(\mathcal{T}_h)$ &  $[\mathcal{P}_r(\mathcal{T}_h)]^3$ & $[\mathcal{P}_{r-1}(\mathcal{T}_h)]^3$  \\
             2         & $[\mathcal{P}_{r+1}(\mathcal{T}_h)]^3$  & $[\mathcal{P}_r(\mathcal{T}_h)]^3$  & $\mathcal{P}_{r-1}(\mathcal{T}_h)$   \\
               3       & $[\mathcal{P}_{r+1}(\mathcal{T}_h)]^3$ &  $\mathcal{P}_r(\mathcal{T}_h)$ & N/A \\
\hline
\end{tabular}
\end{table}

\begin{table}[tbhp]
\caption{Check spaces for DG spaces in Table \ref{tableDGI}}
\label{tablecheckI}
\centering
\begin{tabular}{|c|c|c|c|c|}
\hline
{$k$} 
 & {$\check{\mathcal{V}}^-_h$}  & {$\check{\mathcal{V}}_h$} & {$\check{\mathcal{V}}^\star_h$} & {$\check{\mathcal{V}}_h^{+\star}$} \\

\hline
               0   & N/A & $\mathcal{P}_{r+1}(\mathcal{E}_h)$ & N/A & $\mathring{\mathcal{P}}_{r+1}(\mathcal{E}_h)$ \\
             1       &  $\mathcal{P}_{r+1}(\mathcal{E}_h)$ &  $\mathcal{P}^\parallel_{r}(\mathcal{E}_h)$ &   $\mathring{\mathcal{P}}_{r+1}(\mathcal{E}_h)$ &   $\mathring{\mathcal{P}}^\parallel_r(\mathcal{E}_h)$ \\
             2         &  $\mathcal{P}^\parallel_{r+1}(\mathcal{E}_h)$ &  $\mathcal{P}_r(\mathcal{E}_h)$ & $\mathring{\mathcal{P}}^\parallel_{r+1}(\mathcal{E}_h)$ &  $\mathring{\mathcal{P}}_r(\mathcal{E}_h)$  \\
               3       &  $\mathcal{P}_{r+1}(\mathcal{E}_h)$ &  N/A & $\mathring{\mathcal{P}}_{r+1}(\mathcal{E}_h)$ & N/A \\
\hline
\end{tabular}
\end{table}

\begin{table}[tbhp]
\caption{DG spaces using incomplete polynomials}
\label{tableDGII}
\centering
\begin{tabular}{|c|c|c|c|}
\hline
{$k$} &  {$\mathcal{V}_h^-$}
 & {$\mathcal{V}_h$}  
& {$\mathcal{V}_h^+$} \\

\hline
               0    & N/A &   $\mathcal{P}_{r+1}(\mathcal{T}_h)$ & $[\mathcal{P}_r(\mathcal{T}_h)]^3\times\bm{x}+\mathcal{P}_r(\mathcal{T}_h)$\\
             1       &   $\mathcal{P}_{r+1}(\mathcal{T}_h)$ &  $[\mathcal{P}_r(\mathcal{T}_h)]^3\times\bm{x}+\mathcal{P}_r(\mathcal{T}_h)$ & $[\mathcal{P}_r(\mathcal{T}_h)]^3\bm{x}+\mathcal{P}_r(\mathcal{T}_h)$  \\
             2         & $[\mathcal{P}_r(\mathcal{T}_h)]^3\times\bm{x}+\mathcal{P}_r(\mathcal{T}_h)$ & $[\mathcal{P}_r(\mathcal{T}_h)]^3\bm{x}+\mathcal{P}_r(\mathcal{T}_h)$ & $\mathcal{P}_r(\mathcal{T}_h)$   \\
               3       & $[\mathcal{P}_r(\mathcal{T}_h)]^3\bm{x}+\mathcal{P}_r(\mathcal{T}_h)$ &  $\mathcal{P}_r(\mathcal{T}_h)$ & N/A \\
\hline
\end{tabular}
\end{table}

\begin{table}[tbhp]
\caption{Check spaces for DG spaces in Table \ref{tableDGII}}
\label{tablecheckII}
\centering
\begin{tabular}{|c|c|c|c|c|}
\hline
{$k$} 
 & {$\check{\mathcal{V}}^-_h$}  & {$\check{\mathcal{V}}_h$} & {$\check{\mathcal{V}}^\star_h$} & {$\check{\mathcal{V}}_h^{+\star}$} \\

\hline
               0   & N/A & $\mathcal{P}_{r+1}(\mathcal{E}_h)$ & N/A & $\mathring{\mathcal{P}}_{r+1}(\mathcal{E}_h)$ \\
             1       &  $\mathcal{P}_{r+1}(\mathcal{E}_h)$ &  $\mathcal{P}^\parallel_{r}(\mathcal{E}_h)$ &   $\mathring{\mathcal{P}}_{r+1}(\mathcal{E}_h)$ &   $\mathring{\mathcal{P}}^\parallel_r(\mathcal{E}_h)$ \\
             2         &  $[\mathcal{P}_r(\mathcal{E}_h)]^3$ &  $\mathcal{P}_r(\mathcal{E}_h)$ & $\mathring{\mathcal{P}}^\parallel_r(\mathcal{E}_h)$ &  $\mathring{\mathcal{P}}_r(\mathcal{E}_h)$  \\
               3       &  $\mathcal{P}_r(\mathcal{E}_h)$ &  N/A & $\mathring{\mathcal{P}}_r(\mathcal{E}_h)$ & N/A \\
\hline
\end{tabular}
\end{table}

\textbf{Hodge Laplacian with $\bm{k}$=0.}~With the identification \eqref{identification} in mind, 
the variational Hodge Laplacian \eqref{mixedHodgeLaplace} with $k=0$ is to find $u\in H^1(\Omega)$ and $\xi\in [L^2(\Omega)]^3$ satisfying
\begin{equation}
\begin{aligned}
(\xi,\eta)+(\nabla u,\eta)&=0,\quad&&\eta\in [L^2(\Omega)]^3,\\
(\xi,\nabla v)&=-(f,v),&& v\in H^1(\Omega).
\end{aligned}
\end{equation}
In fact, it is an augmented primal formulation of the Poisson equation under the homogeneous Neumann boundary condition. To guarantee existence and uniqueness of the solution, it is required that $(u,1)=(f,1)=0$. In this case,
the corresponding XG method \eqref{XGlocal} reduces to $(\xi_h,\check{\xi}_h,u_h,\check{u}_h)\in \CV_h^+\times\check{\CV}_h^{+\star}\times\CV_h\times\check{\CV}_h$, such that
on each $T\in\mathcal{T}_h,$ 
\begin{subequations}\label{discreteHodgeLaplacek0}
    \begin{align}
       &(\xi_h,\eta_h)_T-(\nabla\cdot\eta_h,u_h)_T+\langle\hat{u}_h,\eta_h\cdot\nu\rangle^s_{\partial T}=0,\\
       &(\xi_h,\nabla v_h)_T-\langle{v_h},\hat{\xi}_h^\star\rangle^s_{\partial T}=-(f,v_h)_T,\\
    &\ab{\check{\xi}_h^\star-\textsf{b}\lr{u_h}}{\check{\eta}_h^\star}=0,\\
    &\ab{\check{u}_h-\textsf{c}\lr{\xi_h\cdot\nu}}{\check{v}_h}=0,
  \end{align}
\end{subequations}
for all $(\eta_h,\check{\eta}_h,v_h,\check{v}_h)\in \CV_h^+\times\check{\CV}_h^{+\star}\times\CV_h\times\check{\CV}_h$. For well-posedness, one could impose the global constraint $(u_h,1)=0.$ The scheme \eqref{discreteHodgeLaplacek0} recovers the $H(\text{grad})$ based XG method for Poisson's equation in \cite{HongWuXu2021}.

\textbf{Hodge Laplacian with $\bm{k}$=1.}~The Hodge Laplacian \eqref{mixedHodgeLaplace} with $k=1$ is to find $\sigma\in H^1(\Omega)$, $u\in H(\text{curl},\Omega)$, $\xi\in [L^2(\Omega)]^3$ satisfying
\begin{equation}\label{HodgeLaplacek1}
\begin{aligned}
(\sigma,\tau)+(\nabla\tau,u)&=0,\quad&&\tau\in H^1(\Omega),\\
(\xi,\eta)+(\nabla\times u,\eta)&=0,\quad&&\eta\in [L^2(\Omega)]^3,\\
(\nabla\sigma,v)+(\xi,\nabla\times v)&=-(f,v),&& v\in H(\text{curl},\Omega).
\end{aligned}
\end{equation}
Problem \eqref{HodgeLaplacek1} is the variaional formulation of the vector Laplacian 
\begin{align*}
-\nabla(\nabla\cdot u)+\nabla\times\nabla\times u&=f\text{ in }\Omega,\\
u\cdot\nu=0,\quad (\nabla\times u)\times\nu&=0\text{ on }\partial\Omega.
\end{align*}
The XG method \eqref{XGlocal} with $k=1$, $n=3$ seeks
$(\sigma_h,\check{\sigma}_h,\xi_h,\check{\xi}^\star_h,u_h,\check{u}_h,\check{u}^\star_h)\in \CV_h^-\times\check{\CV}_h^-\times\CV_h^+\times\check{\CV}_h^{+\star}\times\CV_h\times\check{\CV}_h\times\check{\CV}^\star_h$, such that
on each $T\in\mathcal{T}_h,$ 
\begin{equation}\label{discreteHodgeLaplacek1}
    \begin{aligned}
       &(\sigma_h,\tau_h)_T+(\nabla\tau_h,u_h)_T-\langle{\tau_h},\hat{u}_h^\star\rangle^s_{\partial T}=0,\\
       &(\xi_h,\eta_h)_T+(\nabla\times\eta_h,u_h)_T+\langle \nu\times\hat{u}_h,\eta_h\times\nu\rangle^s_{\partial T}=0,\\
       &-(\sigma_h,\nabla\cdot v_h)_T+(\xi_h,\nabla\times v_h)_T+\langle\hat{\sigma}_h, v_h\cdot\nu\rangle^s_{\partial T}\\
       &\quad-\langle\nu\times(v_h\times\nu),\hat{\xi}_h^\star\rangle^s_{\partial T}=-(f,v_h)_T,\\
    &\ab{\check{\sigma}_h-\textsf{a}\lr{u_h\cdot\nu}}{\check{\tau}_h}=0,\\
    &\ab{\check{\xi}_h^\star-\textsf{b}\nu\times\lr{ u_h\times\nu}}{\check{\eta}_h^\star}=0,\\
    &\ab{\nu\times\check{u}_h-\textsf{c}\lr{\xi_h\times\nu}}{\nu\times\check{v}_h}=0,\\
    &\ab{\check{u}_h^\star-\textsf{d}\lr{\sigma_h}}{\check{v}_h^\star}=0,
  \end{aligned}
\end{equation}
for all $(\tau_h,\check{\tau}_h,\eta_h,\check{\eta}^\star_h,v_h,\check{v}_h,\check{v}_h^\star)\in \CV_h^-\times\check{\CV}_h^-\times\CV_h^+\times\check{\CV}_h^{+\star}\times\CV_h\times\check{\CV}_h\times\check{\CV}_h^\star$. Therefore we obtain a new class of DG methods for the vector Laplacian.

\textbf{Hodge Laplacian with $\bm{k}$=2.}~The variational Hodge Laplacian \eqref{mixedHodgeLaplace} with $k=2$ is to find $\sigma\in H(\text{curl},\Omega)$, $u\in H(\text{div},\Omega)$, $\xi\in L^2(\Omega)$  such that
\begin{equation}\label{HodgeLaplacek2}
\begin{aligned}
(\sigma,\tau)+(\nabla\times\tau,u)&=0,\quad&&\tau\in H(\text{curl},\Omega),\\
(\xi,\eta)+(\nabla\cdot u,\eta)&=0,\quad&&\eta\in L^2(\Omega),\\
(\nabla\times\sigma,v)+(\xi,\nabla\cdot v)&=-(f,v),&& v\in H(\text{div},\Omega).
\end{aligned}
\end{equation}
Problem \eqref{HodgeLaplacek2} is also a variaional formulation of the vector Laplacian 
\begin{align*}
\nabla\times\nabla\times u-\nabla(\nabla\cdot u)&=f\text{ in }\Omega,\\
u\times\nu=0,\quad \nabla\cdot u&=0\text{ on }\partial\Omega
\end{align*}
under different boundary conditions.
The corresponding XG method for \eqref{XGlocal} seeks $(\sigma_h,\check{\sigma}_h,\xi_h,\check{\xi}^\star_h,u_h,\check{u}_h,\check{u}^\star_h)\in \CV_h^-\times\check{\CV}_h^-\times\CV_h^+\times\check{\CV}_h^{+\star}\times\CV_h\times\check{\CV}_h\times\check{\CV}^\star_h$, such that
on each $T\in\mathcal{T}_h,$ 
\begin{equation}\label{discreteHodgeLaplacek2}
    \begin{aligned}
       &(\sigma_h,\tau_h)_T+(\nabla\times\tau_h,u_h)_T-\langle{\nu\times(\tau_h}\times\nu),\hat{u}_h^\star\rangle^s_{\partial T}=0,\\
       &(\xi_h,\eta_h)_T-(\nabla\eta_h,u_h)_T+\langle\hat{u}_h,\eta_h\rangle^s_{\partial T}=0,\\
       &(\sigma_h,\nabla\times v_h)_T+(\xi_h,\nabla\cdot v_h)_T+\langle\nu\times\hat{\sigma}_h,v_h\times\nu\rangle^s_{\partial T}\\
       &\quad-\langle{v_h\cdot\nu},\hat{\xi}_h^\star\rangle^s_{\partial T}=-(f,v_h)_T,\\
    &\ab{\nu\times\check{\sigma}_h-\textsf{a}\lr{u_h\times\nu}}{\nu\times\check{\tau}_h}=0,\\
    &\ab{\check{\xi}_h^\star-\textsf{b}\lr{u_h\cdot\nu}}{\check{\eta}_h^\star}=0,\\
    &\ab{\check{u}_h-\textsf{c}\lr{\xi_h}}{\check{v}_h}=0,\\
    &\ab{\check{u}_h^\star-\textsf{d}\nu\times\lr{\sigma_h\times\nu}}{\check{v}_h^\star}=0,
  \end{aligned}
\end{equation}
for all $(\tau_h,\check{\tau}_h,\eta_h,\check{\eta}^\star_h,v_h,\check{v}_h,\check{v}_h^\star)\in \CV_h^-\times\check{\CV}_h^-\times\CV_h^+\times\check{\CV}_h^{+\star}\times\CV_h\times\check{\CV}_h\times\check{\CV}_h^\star$. The special XG method \eqref{discreteHodgeLaplacek2} is another family of DG methods for the vector Laplacian.

\textbf{Hodge Laplacian with $\bm{k}$=3.}~The variational Hodge Laplacian \eqref{mixedHodgeLaplace} with $k=3$ is to find $\sigma\in H(\text{div},\Omega)$, $u\in L^2(\Omega)$ such that
\begin{equation}
\begin{aligned}
(\sigma,\tau)+(\nabla\cdot\tau,u)&=0,\quad&&\tau\in H(\text{div},\Omega),\\
(\nabla\cdot\sigma,v)&=-(f,v),&& v\in L^2(\Omega).
\end{aligned}
\end{equation}
Clearly it is the mixed formulation of Poisson's equation under the homogeneous Dirichlet boundary condition.
The XG method \eqref{XGlocal} translates into: Find $(\sigma_h,\check{\sigma}_h,u_h,\check{u}^\star_h)\in \CV_h^-\times\check{\CV}_h^-\times\CV_h\times\check{\CV}^\star_h$, such that
on each $T\in\mathcal{T}_h,$ 
\begin{subequations}
    \begin{align}
       &(\sigma_h,\tau_h)_T+(\nabla\cdot\tau_h,u_h)_T-\langle{\tau_h\cdot\nu},\hat{u}_h^\star\rangle^s_{\partial T}=0,\\
       &-(\sigma_h,\nabla v_h)_T+\langle\hat{\sigma}_h,v_h\rangle^s_{\partial T}=-(f,v_h)_T,\\
    &\ab{\check{\sigma}_h-\textsf{a}\lr{u_h}}{\check{\tau}_h}=0,\\
    &\ab{\check{u}_h^\star-\textsf{d}\lr{\sigma_h\cdot\nu}}{\check{v}_h^\star}=0,
  \end{align}
\end{subequations}
for all $(\tau_h,\check{\tau}_h,v_h,\check{v}_h^\star)\in \CV_h^-\times\check{\CV}_h^-\times\CV_h\times\check{\CV}_h^\star$. This method coincides with the $H(\text{div})$ based XG method for Poisson's equation in \cite{HongWuXu2021}. 

\section*{Appendix}
In the appendix, we give a proof of Lemma \ref{approximation}. We shall construct $v_h^c\in V_h$, $\sigma_h^c\in V^-_h$ using local averaging in Lemma \ref{approximation}. Similar techniques were used for DG spaces containing Lagrange nodal elements and N\'ed\'elec edge elements, see, e.g., \cite{KarakashianPascal2003,HoustonPerugiaSchtzau2005}.
\begin{proof}[Proof of Lemma \ref{approximation}]
First we consider the case $V_h=\mathcal{P}_r\Lambda^k(\mathcal{T}_h)$. Similar arguments apply to $V_h=\mathcal{P}_r^-\Lambda^k(\mathcal{T}_h)$. Let $\Delta(T)$ denote the set of all subsimplexes of $T\in\mathcal{T}_h$ and  $\Delta(\mathcal{T}_h)=\cup_{T\in\mathcal{T}_h}\Delta(T)$. Given a simplex $g\in\Delta(\mathcal{T}_h),$ we fix a basis  $\{\mu_g^i\}_{i\in I_g}$ of $\mathcal{P}_{r+k-\dim(g)}^-\Lambda^{\dim(g)-k}(g)$ such that $\|\mu_g^i\|_{L^\infty(g)}=1$, where $\text{dim}(g)$ is the dimension of $g$. It has been shown in \cite{ArnoldFalkWinther2006,ArnoldFalkWinther2010} that the degrees of freedom of $V_h$ consist of
\begin{equation}\label{dof}
    |g|^{-1}\int_g(\tr_g w_h)\wedge\mu_g^i,\quad g\in\Delta(\mathcal{T}_h),~i\in I_g
\end{equation}
for $w_h\in V_h,$ where $|g|$ is the $\dim(g)$-measure of $g.$ Let $\{\phi_g^i\}_{g\in\Delta(\mathcal{T}_h), i\in I_g}$ be the dual basis of $V_h$ w.r.t. the degrees of freedom in  \eqref{dof}. It holds that
\begin{equation}\label{phiig}
    \|\phi_g^i\|_T\simeq h_T^{\frac{n}{2}},\quad\|d\phi_g^i\|_T\simeq h_T^{\frac{n}{2}-1}.
\end{equation}Let $\omega_g$ denote the set of elements in $\mathcal{T}_h$ sharing $g$ as a common subsimplex. For any $v_h\in\mathcal{V}_h,$ we define the average of $\tr v_h$ surrounding $g$ as
\begin{equation}
    \bar{v}^g_{h,i}=\frac{1}{\#\omega_g}\sum_{T^\prime\in\omega_g}|g|^{-1}\int_g\tr_g(v_h|_{T^\prime})\wedge\mu_g^i.
\end{equation}
The conforming approximation $v_h^c\in V_h$ is constructed as
\begin{equation}\label{vhc2}
    v_h^c=\sum_{g\in\Delta(\mathcal{T}_h)}\sum_{i\in I_g}\bar{v}^g_{h,i}\phi_g^i.
\end{equation}
On each element $T\in\mathcal{T}_h,$ the function $v_h|_T$ is expressed as
\begin{equation}\label{vh}
    v_h|_T=\sum_{g\in\Delta(\mathcal{T}_h)}\sum_{i\in I_g}\left(|g|^{-1}\int_g\tr_g(v_h|_T)\wedge\mu_g^i\right)\phi_g^i.
\end{equation} 
It then follows from \eqref{phiig}, \eqref{vhc2}, \eqref{vh} that 
\begin{equation}\label{intermediateestimate}
\begin{aligned}
    &h^{-1}_T\|(v_h-v_h^c)\|_T+\|d(v_h-v_h^c)\|_T\\
    &\lesssim h_T^{\frac{n}{2}-1}\sum_{g\subset\Delta(T)}\sum_{i\in I_g}\sum_{T^\prime\in\omega_g}\left||g|^{-1}\int_g\{\tr_g(v_h|_T)-\tr_g(v_h|_{T^\prime})\}\wedge\mu_g^i\right|.
\end{aligned}
\end{equation}
Let $\mathcal{E}_h(g)=\{E\in\mathcal{E}_h: \bar{E}\cap\bar{g}\neq\emptyset\}$. 
If $T$ and $T^\prime\in\omega_g$ share a common face $E$ in $\mathcal{E}_h(g)$, we have 
\begin{equation}\label{trTT}
    \|\tr_g(v_h|_T)-\tr_g(v_h|_{T^\prime})\|_{L^\infty(g)}\lesssim h_E^{-\frac{n}{2}+\frac{1}{2}}\|\lr{\tr v_h}\|_E.
\end{equation}
If $T$ and $T^\prime\in\omega_g$ share some low dimensional simplex $g^\prime$ with $\text{dim}(g)<\text{dim}(g^\prime)<n-1$, one could find a chain of pairs of adjacent elements in $\omega_g$ sharing a $\dim(g^\prime)+1$ dimensional simplex in $\omega_g$. Therefore, for an arbitrary $T^\prime\in\omega_g$, there exists a chain of pairs of adjacent elements (possibly repeated) in $\omega_g$ that share a face in $\mathcal{E}_h(g)$. Therefore applying \eqref{trTT} and the triangle inequality, we obtain
\begin{equation}\label{dofdifference}
    \|\tr_g(v_h|_T)-\tr_g(v_h|_{T^\prime})\|_{L^\infty(g)}\lesssim\sum_{E\in\mathcal{E}_h(g)}h_E^{-\frac{n}{2}+\frac{1}{2}}\|\lr{\tr v_h}\|_E,\quad\forall T^\prime\in\mathcal{E}_h(g).
\end{equation}
We also refer to \cite{KarakashianPascal2003} for a detailed argument in nodal DG methods.
Combining \eqref{intermediateestimate} and \eqref{dofdifference}, we obtain
\begin{equation}\label{localT}
    h^{-1}_T\|(v_h-v_h^c)\|_T+\|d(v_h-v_h^c)\|_T\lesssim \sum_{E\in\mathcal{E}_h(g)}h_E^{-\frac{1}{2}}\|\lr{\tr v_h}\|_E.
\end{equation}
Summing \eqref{localT} over all $T\in\mathcal{T}_h$ finishes the proof.

\end{proof}

\bibliography{totalref,DG}

\providecommand{\bysame}{\leavevmode\hbox to3em{\hrulefill}\thinspace}
\providecommand{\MR}{\relax\ifhmode\unskip\space\fi MR }
\providecommand{\MRhref}[2]{%
  \href{http://www.ams.org/mathscinet-getitem?mr=#1}{#2}
}
\providecommand{\href}[2]{#2}
\begin{thebibliography}{10}

\bibitem{arnold1982interior}
Douglas~N Arnold, \emph{An interior penalty finite element method with
  discontinuous elements}, SIAM J. Numer. Anal. \textbf{19} (1982), no.~4,
  742--760.

\bibitem{Arnold2018}
Douglas~N. Arnold, \emph{Finite element exterior calculus}, CBMS-NSF Regional
  Conference Series in Applied Mathematics, vol.~93, Society for Industrial and
  Applied Mathematics (SIAM), Philadelphia, PA, 2018. \MR{3908678}

\bibitem{ArnoldAwanou2014}
Douglas~N. Arnold and Gerard Awanou, \emph{Finite element differential forms on
  cubical meshes}, Math. Comp. \textbf{83} (2014), no.~288, 1551--1570.
  \MR{3194121}

\bibitem{ArnoldAwanouWinther2008}
Douglas~N. Arnold, Gerard Awanou, and Ragnar Winther, \emph{Finite elements for
  symmetric tensors in three dimensions}, Math. Comp. \textbf{77} (2008),
  no.~263, 1229--1251.

\bibitem{ArnoldBrezziCockburnMarini2001}
Douglas~N. Arnold, Franco Brezzi, Bernardo Cockburn, and L.~Donatella Marini,
  \emph{Unified analysis of discontinuous {G}alerkin methods for elliptic
  problems}, SIAM J. Numer. Anal. \textbf{39} (2001/02), no.~5, 1749--1779.
  \MR{1885715}

\bibitem{ArnoldChen2017}
Douglas~N. Arnold and Hongtao Chen, \emph{Finite element exterior calculus for
  parabolic problems}, ESAIM Math. Model. Numer. Anal. \textbf{51} (2017),
  no.~1, 17--34. \MR{3600999}

\bibitem{ArnoldFalkWinther2006}
Douglas~N. Arnold, Richard~S. Falk, and Ragnar Winther, \emph{Finite element
  exterior calculus, homological techniques, and applications}, Acta Numer.
  \textbf{15} (2006), 1--155. \MR{2269741}

\bibitem{ArnoldFalkWinther2007}
\bysame, \emph{Mixed finite element methods for linear elasticity with weakly
  imposed symmetry}, Math. Comp. \textbf{76} (2007), no.~260, 1699--1723.
  \MR{2336264}

\bibitem{ArnoldFalkWinther2010}
\bysame, \emph{Finite element exterior calculus: from {H}odge theory to
  numerical stability}, Bull. Amer. Math. Soc. (N.S.) \textbf{47} (2010),
  no.~2, 281--354. \MR{2594630}

\bibitem{aubin1970approximation}
Jean-Pierre Aubin, \emph{Approximation des problemes aux limites non homogenes
  pour des op{\'e}rateurs non lin{\'e}aires}, Journal of Mathematical Analysis
  and Applications \textbf{30} (1970), no.~3, 510--521.

\bibitem{AwanouFabienGuzmanStern2020}
Gerard Awanou, Maurice Fabien, Johnny Guzm\'an, and Ari Stern,
  \emph{hybridization and postprocessing in finite element exterior calculus},
  arXiv e-prints (arXiv:2008.00149, 2020).

\bibitem{babuvska1973nonconforming}
Ivo Babu{\v{s}}ka and Milo\v{s} Zl{\'a}mal, \emph{Nonconforming elements in the
  finite element method with penalty}, SIAM J. Numer. Anal. \textbf{10} (1973),
  no.~5, 863--875.

\bibitem{BeckerBurmanHansboLarson2003}
R.~Becker, E.~Burman, P.~Hansbo, and M.~G. Larson, \emph{A reduced
  {P}1-{D}iscontinuous {G}alerkin method}, Chalmers Finite Element Center,
  G\"oteborg, Sweden (2003), preprint.

\bibitem{BoffiBrezziFortin2013}
Daniele Boffi, Franco Brezzi, and Michel Fortin, \emph{Mixed finite element
  methods and applications}, Springer Series in Computational Mathematics,
  vol.~44, Springer, Heidelberg, 2013. \MR{3097958}

\bibitem{Brezzi1974}
F.~Brezzi, \emph{On the existence, uniqueness and approximation of saddle-point
  problems arising from {L}agrangian multipliers}, Rev. Fran\c{c}aise Automat.
  Informat. Recherche Op\'{e}rationnelle S\'{e}r. Rouge \textbf{8} (1974),
  no.~{\rm R}-2, 129--151. \MR{365287}

\bibitem{brezzi2000discontinuous}
Franco Brezzi, Gianmarco Manzini, Donatella Marini, Paola Pietra, and
  Alessandro Russo, \emph{{Discontinuous Galerkin approximations for elliptic
  problems}}, Numerical Methods for Partial Differential Equations \textbf{16}
  (2000), no.~4, 365--378.

\bibitem{CarreroCockburnSchotzau2006}
Jes\'{u}s Carrero, Bernardo Cockburn, and Dominik Sch\"{o}tzau,
  \emph{Hybridized globally divergence-free {LDG} methods. {I}. {T}he {S}tokes
  problem}, Math. Comp. \textbf{75} (2006), no.~254, 533--563. \MR{2196980}

\bibitem{ChenWu2017}
Long Chen and Yongke Wu, \emph{Convergence of adaptive mixed finite element
  methods for the {H}odge {L}aplacian equation: without harmonic forms}, SIAM
  J. Numer. Anal. \textbf{55} (2017), no.~6, 2905--2929.

\bibitem{ChristiansenHu2018}
Snorre~H. Christiansen and Kaibo Hu, \emph{Generalized finite element systems
  for smooth differential forms and {S}tokes' problem}, Numer. Math.
  \textbf{140} (2018), no.~2, 327--371. \MR{3851060}

\bibitem{ChristiansenWinther2008}
Snorre~H. Christiansen and Ragnar Winther, \emph{Smoothed projections in finite
  element exterior calculus}, Math. Comp. \textbf{77} (2008), no.~262,
  813--829.

\bibitem{CockburnGopalakrishnanLazarov2009}
Bernardo Cockburn, Jayadeep Gopalakrishnan, and Raytcho Lazarov, \emph{Unified
  hybridization of discontinuous {G}alerkin, mixed, and continuous {G}alerkin
  methods for second order elliptic problems}, SIAM J. Numer. Anal. \textbf{47}
  (2009), no.~2, 1319--1365. \MR{2485455}

\bibitem{cockburn2000development}
Bernardo Cockburn, George~E. Karniadakis, and Chi-Wang Shu, \emph{The
  development of discontinuous {G}alerkin methods}, Discontinuous Galerkin
  Methods, Springer, 2000, pp.~3--50.

\bibitem{CockburnShu1989}
Bernardo Cockburn and Chi-Wang Shu, \emph{T{VB} {R}unge-{K}utta local
  projection discontinuous {G}alerkin finite element method for conservation
  laws. {II}. {G}eneral framework}, Math. Comp. \textbf{52} (1989), no.~186,
  411--435. \MR{983311}

\bibitem{cockburn1998local}
\bysame, \emph{{The local discontinuous Galerkin method for time-dependent
  convection-diffusion systems}}, SIAM Journal on Numerical Analysis
  \textbf{35} (1998), no.~6, 2440--2463.

\bibitem{da2017virtual}
L.~Beirao Da~Veiga, Franco Brezzi, Franco Dassi, L.~Donatella Marini, and
  Alessandro Russo, \emph{Virtual element approximation of 2d magnetostatic
  problems}, Computer Methods in Applied Mechanics and Engineering \textbf{327}
  (2017), 173--195.

\bibitem{Demlow2017}
Alan Demlow, \emph{Convergence and quasi-optimality of adaptive finite element
  methods for harmonic forms}, Numer. Math. \textbf{136} (2017), no.~4,
  941--971.

\bibitem{DemlowHirani2014}
Alan Demlow and Anil~N. Hirani, \emph{A posteriori error estimates for finite
  element exterior calculus: the de {R}ham complex}, Found. Comput. Math.
  \textbf{14} (2014), no.~6, 1337--1371.

\bibitem{DiPietroDroniouRapetti2020}
Daniele~A. Di~Pietro, J\'{e}r\^{o}me Droniou, and Francesca Rapetti,
  \emph{Fully discrete polynomial de {R}ham sequences of arbitrary degree on
  polygons and polyhedra}, Math. Models Methods Appl. Sci. \textbf{30} (2020),
  no.~9, 1809--1855. \MR{4151796}

\bibitem{douglas1976interior}
Jim Douglas and Todd Dupont, \emph{{Interior penalty procedures for elliptic
  and parabolic Galerkin methods}}, Computing Methods in Applied Sciences,
  Springer, 1976, pp.~207--216.

\bibitem{FalkWinther2014}
Richard Falk and Ragnar Winther, \emph{Local bounded cochain projections},
  Math. Comp. \textbf{83} (2014), no.~290, 2631--2656.

\bibitem{FalkNeilan2013}
Richard~S. Falk and Michael Neilan, \emph{Stokes complexes and the construction
  of stable finite elements with pointwise mass conservation}, SIAM J. Numer.
  Anal. \textbf{51} (2013), no.~2, 1308--1326. \MR{3045658}

\bibitem{GilletteHolstZhu2017}
Andrew Gillette, Michael Holst, and Yunrong Zhu, \emph{Finite element exterior
  calculus for evolution problems}, J. Comput. Math. \textbf{35} (2017), no.~2,
  187--212. \MR{3623353}

\bibitem{GilletteKloefkorn2019}
Andrew Gillette and Tyler Kloefkorn, \emph{Trimmed serendipity finite element
  differential forms}, Math. Comp. \textbf{88} (2019), no.~316, 583--606.
  \MR{3882277}

\bibitem{Hiptmair2002}
R.~Hiptmair, \emph{Finite elements in computational electromagnetism}, Acta
  Numer. \textbf{11} (2002), 237--339. \MR{2009375}

\bibitem{HongHuMaXu2020}
Qingguo Hong, Jun Hu, Limin Ma, and Jinchao Xu, \emph{An extended {G}alerkin
  analysis for linear elasticity with strongly symmetric stress tensor}, arXiv
  preprint (2020), arXiv:2002.11664.

\bibitem{hong2017unified}
Qingguo Hong, Fei Wang, Shuonan Wu, and Jinchao Xu, \emph{A unified study of
  continuous and discontinuous {G}alerkin methods}, Science China Mathematics
  \textbf{62} (2019), no.~1, 1--32.

\bibitem{HongWuXu2021}
Qingguo Hong, Shuonan Wu, and Jinchao Xu, \emph{An extended {G}alerkin analysis
  for elliptic problems}, Sci. China Math. \textbf{54} (2021).

\bibitem{HoustonPerugiaSchtzau2004}
Paul Houston, Ilaria Perugia, and Dominik Sch\"{o}tzau, \emph{Mixed
  discontinuous {G}alerkin approximation of the {M}axwell operator}, SIAM J.
  Numer. Anal. \textbf{42} (2004), no.~1, 434--459. \MR{2051073}

\bibitem{HoustonPerugiaSchtzau2005}
\bysame, \emph{Mixed discontinuous {G}alerkin approximation of the {M}axwell
  operator: non-stabilized formulation}, J. Sci. Comput. \textbf{22/23} (2005),
  315--346. \MR{2142200}

\bibitem{KarakashianPascal2003}
Ohannes~A. Karakashian and Frederic Pascal, \emph{A posteriori error estimates
  for a discontinuous {G}alerkin approximation of second-order elliptic
  problems}, SIAM J. Numer. Anal. \textbf{41} (2003), no.~6, 2374--2399.
  \MR{2034620}

\bibitem{LasaintRaviart1974}
P.~Lasaint and P.-A. Raviart, \emph{On a finite element method for solving the
  neutron transport equation}, Mathematical aspects of finite elements in
  partial differential equations ({P}roc. {S}ympos., {M}ath. {R}es. {C}enter,
  {U}niv. {W}isconsin, {M}adison, {W}is., 1974), 1974, pp.~89--123. Publication
  No. 33. \MR{0658142}

\bibitem{li2013hybridizable}
Liang Li, Stephane Lanteri, and Ronan Perrussel, \emph{A hybridizable
  discontinuous {G}alerkin method for solving {3D} time-harmonic {M}axwell’s
  equations}, Numerical Mathematics and Advanced Applications 2011, Springer,
  2013, pp.~119--128.

\bibitem{Li2019SINUM}
Yuwen Li, \emph{Some convergence and optimality results of adaptive mixed
  methods in finite element exterior calculus}, SIAM J. Numer. Anal.
  \textbf{57} (2019), no.~4, 2019--2042. \MR{3995302}

\bibitem{Li2020MCOM}
Yuwen {Li}, \emph{{Quasi-optimal adaptive mixed finite element methods for
  controlling natural norm errors}}, Math.~Comp. \textbf{90} (2021), 565--593.

\bibitem{Licht2019}
Martin~W. Licht, \emph{Smoothed projections and mixed boundary conditions},
  Math. Comp. \textbf{88} (2019), no.~316, 607--635. \MR{3882278}

\bibitem{Lions1968}
J.-L. Lions, \emph{Probl\`emes aux limites non homog\`enes \`a don\'{e}es
  irr\'{e}guli\`eres: {U}ne m\'{e}thode d'approximation}, Numerical {A}nalysis
  of {P}artial {D}ifferential {E}quations ({C}.{I}.{M}.{E}. 2 {C}iclo, {I}spra,
  1967), Edizioni Cremonese, Rome, 1968, pp.~283--292. \MR{0245220}

\bibitem{MuWangYeZhang2015}
Lin Mu, Junping Wang, Xiu Ye, and Shangyou Zhang, \emph{A weak {G}alerkin
  finite element method for the {M}axwell equations}, J. Sci. Comput.
  \textbf{65} (2015), no.~1, 363--386. \MR{3394450}

\bibitem{Nedelec1980}
J.-C. N\'ed\'elec, \emph{Mixed finite elements in $\mathbf{R}^3$}, Numer. Math.
  \textbf{35} (1980), no.~3, 315--341.

\bibitem{Nedelec1986}
\bysame, \emph{A new family of mixed finite elements in $\mathbf{R}^3$}, Numer.
  Math. \textbf{50} (1986), no.~1, 57--81.

\bibitem{NguyenPeraireCockburn2011}
N.~C. Nguyen, J.~Peraire, and B.~Cockburn, \emph{Hybridizable discontinuous
  {G}alerkin methods for the time-harmonic {M}axwell's equations}, J. Comput.
  Phys. \textbf{230} (2011), no.~19, 7151--7175. \MR{2822937}

\bibitem{Nitsche1971}
J.~Nitsche, \emph{\"{U}ber ein {V}ariationsprinzip zur {L}\"{o}sung von
  {D}irichlet-{P}roblemen bei {V}erwendung von {T}eilr\"{a}umen, die keinen
  {R}andbedingungen unterworfen sind}, Abh. Math. Sem. Univ. Hamburg
  \textbf{36} (1971), 9--15. \MR{341903}

\bibitem{Nitsche1972}
\bysame, \emph{On {D}irichlet problems using subspaces with nearly zero
  boundary conditions}, The mathematical foundations of the finite element
  method with applications to partial differential equations ({P}roc.
  {S}ympos., {U}niv. {M}aryland, {B}altimore, {M}d., 1972), 1972, pp.~603--627.
  \MR{0426456}

\bibitem{RaviartThomas1977}
P.-A. Raviart and J.~M. Thomas, \emph{A mixed finite element method for 2nd
  order elliptic problems}, Mathematical aspects of finite element methods
  ({P}roc. {C}onf., {C}onsiglio {N}az. delle {R}icerche ({C}.{N}.{R}.), {R}ome,
  1975), 1977, pp.~292--315. Lecture Notes in Math., Vol. 606. \MR{0483555}

\bibitem{ReedHill1973}
W.~H. Reed and T.~R. Hill, \emph{Triangular mesh methods for the neutron
  transport equation}, Los Alamos Scientific Laboratory, Los Alamos, NM (1973),
  Tech. Report LA--UR--73--479.

\bibitem{SunLiu2009}
Shuyu Sun and Jiangguo Liu, \emph{A locally conservative finite element method
  based on piecewise constant enrichment of the continuous {G}alerkin method},
  SIAM J. Sci. Comput. \textbf{31} (2009), no.~4, 2528--2548. \MR{2520288}

\bibitem{wheeler1978elliptic}
Mary~Fanett Wheeler, \emph{An elliptic collocation-finite element method with
  interior penalties}, SIAM J. Numer. Anal. \textbf{15} (1978), no.~1,
  152--161.

\bibitem{XuZikatanov2003}
Jinchao Xu and Ludmil Zikatanov, \emph{Some observations on {B}abu\v{s}ka and
  {B}rezzi theories}, Numer. Math. \textbf{94} (2003), no.~1, 195--202.
  \MR{1971217}

\end{thebibliography}

\end{document}